\documentclass[12pt]{article}

\usepackage{amssymb,amsthm,amsmath,amsfonts}

\usepackage{xcolor}
\usepackage{multirow}
\usepackage{enumitem}

\newcommand{\comment}[1]{}

\definecolor{darkgreen}{RGB}{34, 175, 65}
\definecolor{nicepink}{RGB}{255,20,140}

\usepackage[normalem]{ulem}

\newtheorem{theorem}{Theorem}[section]
\newtheorem{lemma}[theorem]{Lemma}
\newtheorem{corollary}[theorem]{Corollary}

\theoremstyle{definition}

\newtheorem{example}[theorem]{Example}
\newtheorem{question}[theorem]{Question}

\newcommand{\TD}{{\rm TD}}
\newcommand{\PD}{{\rm PD}}

\newcommand{\Integer}{\mathbb Z}

\newcommand{\collection}{\mathbb S}
\newcommand{\floor}[1]{\lfloor{#1}\rfloor}

\def \cB {{\cal B}}

\def \cD {{\cal D}}

\def \Z {\mathbb Z}

\def \gdd {\mathrm{GDD}}
\def \bibd {\mathrm{BIBD}}
\def \sts {\mathrm{STS}}

 \title {Colourings of Uniform Group Divisible Designs and Maximum Packings}

\begin{document}
\maketitle

\begin{center}Andrea C. Burgess$^1$,
Peter Danziger$^2$,
Diane Donovan$^3$,
Tara Kemp$^3$,\\
James G. Lefevre$^3$,
David A. Pike$^4$, E. \c{S}ule Yaz{\i}c{\i}$^5$
\end{center}

\noindent$^1$ Department of Mathematics and Statistics, University of New Brunswick, Saint John, NB, E2L 4L5, Canada\\
$^2$ Department of Mathematics, Toronto Metropolitan  University, Toronto, ON, M5B 2K3, Canada\\
$^3$ School of Mathematics and Physics, ARC Centre of Excellence, Plant Success in Nature and Agriculture, University of Queensland, Brisbane, 4072, Australia\\
$^4$ Department of Mathematics and Statistics, Memorial University of Newfoundland, St.~John's, NL, A1C 5S7, Canada\\
$^5$ 
Mathematics Department, Ko\c{c} University, Sar\i yer, Istanbul, T\"{u}rkiye\\

\begin{abstract}
A weak $c$-colouring of a design is an assignment of colours to its  points from a set of $c$ available colours, such that there are no monochromatic blocks. 
A colouring of a design is block-equitable, if for each block, the number of points coloured with any available pair of colours differ by at most one.
Weak and block-equitable colourings of balanced incomplete block designs have been previously considered.
In this paper, we extend these concepts to group divisible designs (GDDs) and packing designs.
We first determine when a $k$-GDD of type $g^u$ can have a block-equitable $c$-colouring. We then give a direct construction of maximum block-equitable $2$-colourable packings with block size $4$; a recursive construction has previously appeared in the literature.
We also generalise a bound given in the literature for the maximum size of block-equitably $2$-colourable packings to $c>2$. Furthermore, we establish the asymptotic existence of uniform $k$-GDDs with arbitrarily many groups and arbitrary chromatic numbers (with the exception of $c=2$ and $k=3$). A structural analysis of $2$- and $3$-uniform $3$-GDDs obtained from 4-chromatic STS$(v)$ where $v\in\{21,25,27,33,37,39\}$ is given.  
We briefly discuss weak colourings of packings,
and finish by considering some further constraints on weak colourings of GDDs, namely requiring all groups to be either monochromatic or equitably coloured.
\end{abstract}


\section{Introduction}

Interactions between combinatorial design theory and other branches of mathematics are well documented, with synergies to finite geometry, number theory, linear algebra, graph theory, universal algebra, particularly group theory and finite field theory. These synergies have led to a diversity of applications, for instance in experimental design, broad-spectrum communication, network theory, coding theory and cryptography. 
The assignment of colours to elements of block designs has enriched studies in combinatorial design theory leading to interesting theoretical problems, as well as novel proofs and increased applications. The  Handbook of Combinatorial Designs~\cite{Handbook} provides such examples, for instance in Chapter~VII, Table~6.4, it notes that the following questions have been shown to be NP-hard: {\em ``Is a block design $t$-colorable, for all $t \geq 9$?''}  and {\em ``Is a partial STS $t$-colorable for any fixed $t \geq 3$?''} \cite{Colbourn1982A}, {\em ``Is a partial SQS $2$-colorable?''} \cite{Colbourn1982B}, {\em ``Is an STS $14$-colorable?''} and  
{\em ``Is an STS $k$-chromatic?''} \cite{Phelps1984}. Further, Woolbright~\cite{Woolbright1978} and others have used ``rainbows'' in the study of partial transversals in Latin squares. 

A recent application of colouring of points in a block design can be found in DNA based storage systems \cite{Tabatabaei2020}. In this storage paradigm, data is encoded as binary sequences that can be recorded as {\em ``nicks''} -- in vitro topological modifications of native DNA such as {\em E.~coli\/} DNA. The data sequences are stored across specified loci in the sugar-phosphate backbone as opposed to previous storage methods via nucleotide sequences. Tabatabaei et al., \cite{Tabatabaei2020}, argue that such nick-based DNA storage is well-suited to compressed data and allows for the erasure of metadata if needed. Given this potential, researchers have investigated efficient methods for encoding the data on the DNA.  Gabrys et al.\ \cite{GDCM2020} note that for efficient error correction while ensuring the stability of the backbone, the positioning of nicks should be determined by small overlapping (intersecting) sets which contain a (near) equal number of positions from each strand. Block designs, and in particular, colouring the points of block designs, provide a good framework for determining the nick positions. In Subsection \ref{Section-eq_pack} we review the colouring aspects of the research of Gabrys et al., but leave the reader to refer to \cite{GDCM2020, Tabatabaei2020} for more details on DNA data storage. However, one question that arises is whether these colouring techniques can be translated back into the design of gene-based experiments, particularly those involving the modification of SNPs.

We now turn to the notation and relevant definitions for block designs and colourings. Firstly, $[n]$ will be used to denote the set of integers $\{1,\dots,n\}$.
A {\em block design} is a pair ${\cal D}=(V, {\cal B})$, where $V$ is a finite set of points, of size $v$, and ${\cal B}$ is a collection of subsets of $V$, called {\em blocks}.
A {\em balanced incomplete block design}, denoted  BIBD$(v,k,\lambda)$, or just BIBD, is a block design where all blocks have size $k$ and every pair of points occurs in exactly $\lambda$ blocks.
A {\em parallel class}, $\pi$, of a BIBD$(v,k,\lambda)$ is a set of $v/k$ blocks that partition the point set $V$.
Simple counting arguments verify that if a BIBD$(v,k,\lambda)$ exists then
\begin{itemize}
   \item[1.] $\lambda(v-1)\equiv 0 \pmod{k-1}$, and
   \item[2.] $\lambda v(v-1)\equiv 0 \pmod{k(k-1)}.$
\end{itemize}
We say $v$  is {\em $(k,\lambda)$-admissible} if Conditions 1 and 2 are met.
This definition can be generalised to a $t$-design where every $t$-subset of points occurs in exactly $\lambda$ blocks. In the main, this article is restricted to discussions where $t=2$ and often $\lambda=1$. It will also be assumed that  $|{\cal B}|\geq 1$ and $v\geq k\geq 2$.
Steiner triple systems (STS$(v)$, $v\geq 7$) are a well studied class of BIBDs with $k=3$ and $\lambda=1$, with
 admissible values $v\equiv 1,3 \pmod{6}$; see~\cite{CR1999} for more details about triple systems.

When $v,\ k$ and $\lambda$ 
do not satisfy Conditions 1 and 2,  we may define a {\em Packing Design}, denoted \PD$(v, k, \lambda)$ 
to be a block design ${\cal D}=(V,{\cal B})$ with blocks of size $k$, chosen such that every pair of points occurs in at most $\lambda$ blocks of $\cB$. The {\em leave} of a \PD$(v, k, 1)$ is defined to be the graph on vertex set $V$ such that $\{x,y\}$ is an edge if and only if the points $x,y$ do not occur together in any block of ${\cal B}$. The number of blocks in a packing is said to be its {\em size}. In contrast to BIBDs, for fixed $k$ and $\lambda$, packings exist for all $v$ and the interest is typically in finding the maximum possible size of a \PD$(v, k, \lambda)$.

A {\em group divisible design (GDD)} is defined by a triple ${\cal D}=(V, {\cal G}, {\cal B})$, where $V$ is a finite set of points, ${\cal G}$ is a partition of $V$ into {\em groups}, and ${\cal B}$ is a collection of subsets of $V$, called {\em blocks}, chosen such that every pair of points appears in exactly one block or exactly one group, but not both.  A $K$-GDD of type $\prod g_i^{u_i}$ is a GDD where the size of each block is an element of a set $K$ and there are exactly $u_i$ groups of size $g_i$, so $|V|=\sum g_iu_i$. If $K=\{k\}$, then we refer to a $k$-GDD and it is said to be {\em uniform} if all groups have equal size, denoted a $k$-GDD of type $g^u$.

A $k$-GDD is equivalent to a \PD$(v, k, 1)$ in which the leave is the union of cliques that partition the point set.  
More generally, we have
 $\lambda$-fold GDDs having each pair of points not in a group occurring in exactly $\lambda$ blocks, denoted $K$-$\gdd_{\lambda}$ or $k$-$\gdd_{\lambda}$.  These objects generalise  BIBD$(v,k,\lambda)$,
which are $k$-$\gdd_{\lambda}$ designs of type $1^v$.
In many cases, the block sizes or group sizes will be taken to be constant and often $\lambda=1$.
Counting arguments verify that if a $k$-GDD$_{\lambda}$ of type $g^u$ exists, then:
\begin{enumerate}
\item[3.] $k \leq u$,
\item[4.] $(k-1) \mid \lambda g(u-1)$, and
\item[5.] $k(k-1) \mid \lambda u(u-1)g^2$.
\end{enumerate}
A pair $(u,g)$ that satisfies Conditions 3, 4 and 5 is said to be {\em $(k,\lambda)$-admissible}.  In the case that $g=1$, we say that $u$ is $(k,\lambda)$-admissible.

A {\em transversal design} TD$(k,g)$ is a $k$-GDD of type $g^k$. 
We note that $\TD(k, g)$ are known to exist for many pairs $(k,g)$, but necessarily $k\leq g+1$. For example, if $g$ is a prime power and $k\leq g+1$, then it is well known that a $\TD(k,g)$ exists. Further, for fixed $k$, there exists an $n_k$ such that a $\TD(k,g)$ exists whenever $g \geq n_k$, and for all but finitely many $k$, $n_k \leq k^{\frac{1}{14.8}}$ \cite{Beth}.  See \cite[Section III.3]{Handbook} for details on these and many other such results on transversal designs.

In this article, we will assume ${\cal D}=(V,{\cal B})$ refers to a block design with constant block size $k$ and $t=2$ unless otherwise stated.  We will colour the points of designs such that the colourings are equitable or weak, in accordance with the following definitions.

Let ${\cal C}$ be a set of cardinality $c\geq 2$.  A {\em $c$-colouring} of   a point set $V$ is a function $f : V \rightarrow {\cal C}$, with level sets $F(\gamma) = \{x\in V \mid f(x) = \gamma \}$,  termed  {\em colour classes}. 
A {\em block-equitable $c$-colouring} of a block design ${\cal D}=(V, {\cal B})$ is a colouring that satisfies
 $\lfloor \frac{k}{c}\rfloor \leq |B\cap F(\gamma)| \leq \lceil \frac{k}{c}\rceil$,  for each $\gamma \in {\cal C}$ and each $B\in \cB$. If such a colouring exists, we say that the block design can be {\em block-equitably} $c$-coloured.

We may also refer to a {\em point-equitable} $c$-colouring of a point set $V$, and take this to mean that the sizes of the colour classes within $V$ differ by at most one.  
Note that if $f : V \rightarrow {\cal C}$ is a point-equitable $c$-colouring and $c \leq |V|$, then $f$ is necessarily surjective. Moreover, a point-equitable colouring of a set $V$ may be viewed as a block-equitable colouring of a design $(V, \{V\})$ with a single block. Additionally, given a block-equitable colouring, we can view it as a point-equitable colouring when restricted to a single block.

Pairs of points of the same (resp.\ different) colour are referred to as {\em monochrome} (resp.\ non-monochrome) pairs.  
Given a block design ${\cal D}=(V,{\cal B})$ (which may be a BIBD, a GDD or a packing), and a colouring  $f : V \rightarrow {\cal C}$,  a block  $B \in \cB$ is said to be {\em monochromatic} if $B\subseteq F(\gamma)$, for some $\gamma \in {\cal C}$.

A {\em weak colouring} of a design is a colouring that satisfies $f(x)\neq f(y)$ for all blocks $B\in \cB$ and some pair $x,y\in B$; that is, no block is monochromatic. A block design that can be weakly coloured with $c$ colours is called {\em weakly $c$-colourable}.  A design  ${\cal D}$ that is weakly $c$-colourable, but not weakly $(c-1)$-colourable is said to have {\em (weak) chromatic number} $\chi(\cD) = c$, or sometimes $\chi=c$.  If $\mathcal{D}$ has chromatic number $c$, then it is called {\em $c$-chromatic.}

\begin{example} \label{Ex:STS(7)}
The colour classes $F(1)=\{{\color{blue}0},{\color{blue}1},{\color{blue}2}\}$, $F(2) = \{{\color{red}\mathit{3}},{\color{red}\mathit{4}}\}$ and $F(3)=\{{\color{darkgreen}\mathbf{5}},{\color{darkgreen}\mathbf{6}}\}$ form a weak $3$-colouring of the following $\bibd(7,3,1)$:
\[
\{{\color{blue}0},{\color{blue}1},{\color{red}\mathit{3}}\}, \{{\color{blue}1},{\color{blue}2},{\color{red}\mathit{4}}\}, \{{\color{blue}2},{\color{red}\mathit{3}},{\color{darkgreen}\mathbf{5}}\}, \{{\color{red}\mathit{3}},{\color{red}\mathit{4}},{\color{darkgreen}\mathbf{6}}\}, \{{\color{red}\mathit{4}},{\color{darkgreen}\mathbf{5}},{\color{blue}0}\}, \{{\color{darkgreen}\mathbf{5}},{\color{darkgreen}\mathbf{6}},{\color{blue}1}\}, \{{\color{darkgreen}\mathbf{6}},{\color{blue}0},{\color{blue}2}\}.
\]
Note that this colouring is point-equitable since each colour class has size $2$ or $3$.  However, it is not block-equitable since the block $B=\{{\color{blue}0},{\color{blue}1},{\color{red}\mathit{3}}\}$ satisfies $|F(1) \cap B| - |F(3) \cap B| = 2$.  
\end{example}

\begin{example}
The $\bibd(6,5,4)$ with block set 
\[
\{{\color{blue}0}, {\color{blue}1}, {\color{red}\mathit{2}}, {\color{red}\mathit{3}}, {\color{darkgreen}\mathbf{4}}\}, 
\{{\color{blue}0}, {\color{blue}1}, {\color{red}\mathit{2}}, {\color{red}\mathit{3}}, {\color{darkgreen}\mathbf{5}}\}, 
\{{\color{blue}0}, {\color{blue}1}, {\color{red}\mathit{2}}, {\color{darkgreen}\mathbf{4}}, {\color{darkgreen}\mathbf{5}}\},
\{{\color{blue}0}, {\color{blue}1}, {\color{red}\mathit{3}}, {\color{darkgreen}\mathbf{4}}, {\color{darkgreen}\mathbf{5}}\},
\{{\color{blue}0}, {\color{red}\mathit{2}}, {\color{red}\mathit{3}}, {\color{darkgreen}\mathbf{4}}, {\color{darkgreen}\mathbf{5}}\},
\{{\color{blue}1}, {\color{red}\mathit{2}}, {\color{red}\mathit{3}}, {\color{darkgreen}\mathbf{4}}, {\color{darkgreen}\mathbf{5}}\}
\]
has a $3$-colouring with colour classes $F(1) = \{ {\color{blue}0}, {\color{blue}1}\}$, $F(2)=\{{\color{red}\mathit{2}}, {\color{red}\mathit{3}}\}$ and  $F(3)=\{{\color{darkgreen}\mathbf{4}}, {\color{darkgreen}\mathbf{5}}\}$.  This colouring is both point-equitable and block-equitable.  To see the latter, note that each block contains one point of one colour and two points of each of the other two. 
\end{example}

While there is a substantial body of knowledge regarding colourings of BIBDs, much less is known about colourings of  GDDs and packings. For completeness, we will review all three types of designs but focus on the latter two, as is the case in Section~\ref{sc:equitable} where we briefly review equitable colourings for  BIBDs and then extend the results for equitable colouring of GDDs and packings. In Theorem~\ref{th:equitableGDD} we show that a given $k$-GDD of type $g^u$ is block-equitably $c$-colourable if and only if (i) $u\leq c\leq ug$, or (ii) $k=u$, or (iii) $k=u-1$ and $c\mid u$. Furthermore, given positive integers $v\geq k\geq 3$ and $c\geq2$, we study the maximum possible size of a packing PD$(v,k,1)$ that has a block-equitable $c$-colouring. We generalise an upper bound  given in~\cite{GDCM2020} for the maximum size of a block-equitably $2$-colourable packing to the case $c>2$ and show that this bound is achievable in some cases; specifically a block-equitable $c$-coloured PD$(v,k,1)$ of maximum possible size exists if $c \mid k$ and a transversal design TD$(k,v/k)$ exists. 

In Section~\ref{sc:weak}, again we briefly review results on weak colourings of BIBDs and then focus on new results for weakly coloured GDDs.  In Corollary~\ref{blowing-up_corollary2} we establish the asymptotic existence of uniform $k$-GDDs with arbitrarily many groups and arbitrary chromatic numbers (with the exception of $c=2$ and $k=3$).  A structural analysis of $2$- and $3$-uniform $3$-GDDs obtained from 4-chromatic STS$(v)$ where $v\in\{21,25,27,33, 37, 39\}$ is given at the end of Section~\ref{ssc:weakGDD}. 
In Section~\ref{Sec:Weakly Coloured Packings}, we discuss the possible chromatic numbers of packings.

Finally, in Section~\ref{OtherColor}, we focus on weak colourings of GDDs with the minimum possible number of colours such that either all groups are monochromatic (see Section~\ref{Sec:Monogroup}) or all groups are equitably coloured (see Section~\ref{Sec:GroupEquitable}). We analyse how the chromatic number changes when we impose the condition that the groups of the GDD are monochromatic. 
In Corollary \ref{cor:groupequitable4gdd}, we show that when $u \geq 4$, $(u-1)g \equiv 0 \pmod{3}$ and $u(u-1)g^2 \equiv 0 \pmod{12}$, there is a $2$-chromatic $4$-GDD of type $(4g)^u$ where the groups are equitably coloured; this result extends to arbitrary block size when certain conditions are met.

\section{Block-Equitable Colourings}
\label{sc:equitable}
We begin with a general relationship between point-equitable $c$-colourings of a set $V=\{x_1,\dots,x_{\mu}\}$ and the minimum number of monochrome pairs of the elements of  $V$.
Specifically, we will show that the number and proportion of the monochrome pairs of elements of $V$  is minimised when the colouring is point-equitable. 

We observe that with respect to a point-equitable $c$-colouring of $V$, given non-negative integers $\alpha,\beta$ satisfying 
 $\mu=\alpha c+\beta$ and $0\leq \beta<c$, it follows 
 that each of $\beta$ colours are assigned to $\alpha+1$ elements of $V$ and each of $c-\beta$ colours are each assigned to $\alpha$ elements of $V$.
 Then for point-equitable $c$-colourings and pairs $x_j,x_{j^\prime}\in V$, counting arguments establish formulae for the number of non-monochrome pairs, $nm_c(\mu)$,  the number of monochrome pairs, $m_c(\mu)$, and the proportion of monochrome pairs with respect to all possible $\binom{\mu}{2}$ pairs, $pm_c(\mu)$:
\begin{eqnarray*}
nm_c(\mu)&=&\alpha^2{c-\beta\choose 2}+(\alpha+1)^2{\beta \choose 2} +\alpha(\alpha+1)\beta(c-\beta)\\
&=& \alpha\left(\frac{\alpha c}{2}+\beta\right)(c-1)+\frac{\beta(\beta-1)}{2},\\
m_c(\mu)&=&\beta{\alpha+1 \choose 2} +(c-\beta){\alpha\choose 2}
=\frac{\alpha(\alpha-1)c+2\alpha\beta}{2},\mbox{ and}\\
pm_c(\mu) &=& \frac{\alpha(\alpha-1)c+2\alpha\beta}{\mu(\mu-1)}.\\
\end{eqnarray*}
Note that $m_c(\mu)+nm_c(\mu)={\mu\choose 2}$.

The following lemma shows that the expression for $pm_c(\mu)$ above yields a bound for the proportion of monochrome pairs in an arbitrary $c$-colouring.

\begin{lemma}\label{lm:Cauchy} Suppose there is a $c$-colouring of a set $V=\{x_1,\dots,x_{\mu}\}$. 
Then the proportion of pairs $x_j,x_{j^\prime}\in V$ that are monochrome is minimised if and only if the elements of $V$ are point-equitably coloured.
\end{lemma}

\begin{proof}
For $i\in [c]$,
let $c_i$ denote the number of elements $x_j$ coloured with colour $i$, so $\sum_i c_i=\mu$. 
The number of monochrome pairs is given by
\begin{eqnarray*}
\sum_i \frac{c_i^2-c_i}{2}=\frac{1}{2}\left(\sum_i c_i^2-\mu\right).
\end{eqnarray*}
Suppose that the elements of $V$ are not point-equitably coloured, so that $c_{i_1}-c_{i_2}>1$, for some $i_1,i_2$. 
Then 
\begin{eqnarray*}
(c_{i_1}-1)^2+(c_{i_2}+1)^2 = c_{i_1}^2+c_{i_2}^2 - 2(c_{i_1}-c_{i_2}-1) < c_{i_1}^2+c_{i_2}^2.
\end{eqnarray*}
Therefore, the number of monochrome pairs can be reduced by recolouring one element from colour $i_1$ to $i_2$. Thus the number of monochrome pairs is minimised exactly when the elements are point-equitably coloured.
\end{proof}

\subsection{Block-Equitably Coloured BIBDs}

A discussion of equitable colourings of BIBDs and in particular STSs is included here for completeness and we note that many of the results on weak colourings of BIBDs (in Subsection \ref{ssc:weakBIBD}) are directly applicable to the existence of equitable colourings.

For instance, if $v \geq 7$, it is known by a simple counting argument that for any STS$(v)$, ${\cal D}$, $\chi({\cal D})> 2$~\cite{Pelikan1970,Rosa1970}, implying that for $v\geq 7$, no block-equitably 2-colourable STS$(v)$ exists.
Adams, Bryant, Lefevre, and Waterhouse~\cite{AdamsBryantLefevreWaterhouse2004}  also established that if
$c\geq  3$,  there exists an equitably $c$-colourable STS$(v)$ if and only if $c = v$ and $v \equiv 1, 3 \pmod{6}$.   This result has been generalised by Luther and Pike in Theorems 1 and 4 of~\cite{LutherPike2016}, establishing necessary and sufficient conditions for the existence of equitably colourable BIBDs:

\begin{theorem}[\cite{LutherPike2016}]
\label{thm:equitableBIBD}
Let $\cD$ be a BIBD$(v,k,\lambda)$. Then there exists a block-equitable $c$-colouring of  $\cD$  if and only if
\begin{itemize}
    \item $v \leq c$ or
    \item $v=k$, or
    \item $v=k+1$, $\lambda \mid (k-1)$, $k>c$ and $c \mid (k+1).$
\end{itemize}
\end{theorem}

Block-equitable colourings have also been studied for $k$-cycle decompositions of the complete graph, which coincide with Steiner triple systems in the case $k=3$.  In particular, the paper~\cite{AdamsBryantLefevreWaterhouse2004} focused on block-equitable colourings of cycle systems.  For more information on block-equitable colourings of cycle decompositions, see also~\cite{AdamsBryantWaterhouse2007, BurgessMerola2021, BurgessMerola2024, LefevreWaterhouse2005, Waterhouse2006}.

\subsection{Block-Equitably Coloured GDDs}
\label{Section-eq_GDD}

We obtain necessary conditions for equitably $c$-coloured uniform $k$-GDDs by considering the proportion of monochrome pairs. Here we refer to pairs of points occurring together in blocks, that is, they belong to different groups. We then show sufficiency of these conditions by construction.

\begin{lemma}
\label{min_mono__group_GDD}
Let ${\cal D}$ be a  $k$-GDD with $u$ groups, and $2 \leq c\leq u$.
Considering all possible $c$-colourings of ${\cal D}$, the minimum number of monochrome pairs summed over all blocks can be achieved using a colouring in which each group $G_i$ is monochromatic.
\end{lemma}

\begin{proof}
Assume that the $c$-colouring of ${\cal D}$  minimises the number of monochrome pairs summed over all blocks of ${\cal D}$ (that is, we consider only pairs of points where the points belong to distinct groups). Consider a group $G_i$, and suppose there exist points $p,p^\prime\in G_i$, where $p$ and $p^{\prime}$ are assigned distinct colours $R$ and $R^\prime$. Let $x_i(R)$ denote the number of points of colour $R$ in $\cup_{j\ne i} G_j$ and assume without loss of generality that $0 \leq x_i(R)\leq x_i(R^{\prime})$.  
Recolouring $p^\prime$ with colour $R$ will change the number of monochrome pairs by $x_i(R)-x_i(R^\prime)\leq 0$. If $x_i(R) <  x_i(R^{\prime})$ this is a contradiction, thus $x_i(R) =  x_i(R^{\prime})$. This equality holds for any $p^\prime\in G_i \setminus \{p\}$, hence all points in $G_i$ may be recoloured to $R$ without changing the number of monochrome pairs. Proceeding in this way for each group $G_i$, we ensure monochromatic groups while retaining the minimum possible number of monochrome pairs. 
\end{proof}

In the following lemma we give two trivial constructions.

\begin{lemma}\label{lm:trivialGDD} Suppose  that ${\cal D}$ is a $k$-GDD of type $g^u$. If $v \geq c\geq u$ or
 $k=u$, then there exists  a block-equitable $c$-colouring of ${\cal D}$.
 \end{lemma}

\begin{proof} 
If $v \geq c\geq u$ the points can be $c$-coloured so that no two points from different groups receive the same colour. The $k$ points of each block are thus assigned $k$ distinct colours, and ${\cal D}$ is block-equitably coloured. 
If $c \leq u$ and $k=u$ then we can colour each group with a single colour, and ${\cal D}$ is a TD$(k,g)$ so that each block contains one point from each group. Thus a block-equitable colouring of ${\cal D}$ is achieved by colouring, for each colour $R$, all vertices in either $\floor{k/c}$ or $\floor{k/c}+1$ groups  with colour $R$.
\end{proof}

\begin{lemma}
\label{min_monochrome_GDD}
Considering all pairs of points in distinct groups of a $c$-coloured $k$-GDD of type $g^u$, the proportion which are monochrome pairs  is at least $pm_c(u)$.
\end{lemma}

\begin{proof}
Let $u=\alpha c + \beta$, for non-negative integers $\alpha,\beta$ and $0\leq \beta<c$.
If $c > u$ then $\alpha =0$ and thus $pm_c(u)=0$, and the result holds. Thus we can assume that $c \leq u$, and hence by Lemma~\ref{min_mono__group_GDD} we can assume without loss of generality that each group is monochromatic. For each pair of groups that are assigned the same colour, we have $g^2$ monochrome pairs. Thus the number of monochrome pairs of points across groups is minimised when the number of pairs of groups with the same colour is minimised. By Lemma~\ref{lm:Cauchy}, with $X_1, \dots ,X_{\mu}$ representing groups of the GDD, this minimum is $m_c(u)$, and is achieved when we colour the set $\{X_1, \ldots, X_{\mu}\}$ point-equitably.  
Thus, as each element of $X_i$ is assigned the same colour as $X_i$, the minimum number of monochrome pairs across groups is $m_c(u) g^2$. Since the total number of pairs of points from distinct groups is $g^2 u(u-1)/2$, the result follows.
\end{proof}

\begin{theorem}\label{th:equitableGDD}
Let ${\cal D}$ be a $k$-GDD of type $g^u$, with $3 \le k \le u$. Then there exists a block-equitable $c$-colouring of ${\cal D}$ if and only if at least one of the following holds:
\begin{itemize}
\item $u\le c \le ug$
\item $k=u$
\item $k=u-1$ and $c \mid u$.
\end{itemize}
\end{theorem}

\begin{proof}
The sufficiency of these conditions is straightforward, with the first two coming directly from Lemma \ref{lm:trivialGDD}.  If $k=u-1$ and $c \mid u$, each colour is applied to all the points from $u/c$ groups.

We now consider the necessity of these conditions. 
Assume there exists a block-equitably $c$-coloured $k$-GDD of type $g^u$,    ${\cal D}=(V,{\cal B})$. Further, assume $3 \le k \le u$ and $k=ac + b$, $0\leq b<c$.
Recall that when restricted to a block $B$, the colouring of $B$ is point-equitable, and for each block $B \in {\cal B}$, the proportion of monochrome pairs in $B$ is exactly
$pm_c(k)=\left( b{a+1 \choose 2} + (c-b){a \choose 2} \right) / {k \choose 2} $.

Since each pair of points from distinct groups occurs in exactly one block, Lemma \ref{min_monochrome_GDD} implies that
$$pm_c(k) \ge pm_c(u).$$ 
Next we consider the function $pm_c(\mu)$; observe that if this function is strictly increasing then $k=u$. Below we show that in fact $pm_c(\mu+1)-pm_c(\mu)>0$, except in the case where $c \mid (\mu +1)$, in which case we have equality.

Let $\mu=\alpha c+\beta$, where $0\leq \beta<c$.

If $c \mid (\mu+1)$, then $\mu+1=zc$ for $z\in \mathbb{N}$ and $\mu=zc-1$ implying $z=\alpha+1$ and $\mu=\alpha c+c-1$. Hence
\begin{eqnarray*}
pm_c(\mu+1)-pm_c(\mu)
&=&  \frac{c(\alpha+1)\alpha}{(\mu+1)\mu} - \frac{c\alpha(\alpha-1)+2\alpha(c-1)}{\mu(\mu-1)} \\
&=& \frac{\alpha[c(\alpha+1)(\mu-1)-c(\alpha-1)(\mu+1)-2(c-1)(\mu+1)]}{\mu(\mu^2-1)} \\
&=& \frac{2\alpha[\mu-(\alpha c+c-1)]}{\mu(\mu^2-1)} \\
&=& 0.
\end{eqnarray*}
If  $\mu+1$ is not divisible by $c$, then $\floor{\frac{\mu+1}{c}}=\alpha $, and $\mu+1=\alpha c+\beta +1$, where $1\leq \beta+1 <c$, or $\beta<c-1$.
Hence
\begin{eqnarray*}
pm_c(\mu+1)-pm_c(\mu)
&=&  \frac{c\alpha (\alpha-1)+2\alpha(\beta+1)}{(\mu+1)\mu} - \frac{c\alpha (\alpha-1)+2\alpha \beta}{\mu(\mu-1)} \\
&=& \frac{\alpha[c(\alpha-1)(\mu-1)+2(\beta+1)(\mu-1)-c(\alpha-1)(\mu+1)-2\beta(\mu+1)]}{\mu(\mu^2-1)} \\
&=& \frac{2\alpha(\mu-\alpha c+c-2\beta-1)}{\mu(\mu^2-1)} \\
&=& \frac{2\alpha(c-\beta-1)}{\mu(\mu^2-1)}.
\end{eqnarray*}
Therefore $pm_c(\mu+1)-pm_c(\mu)>0$ unless $\alpha=0$ (that is, $\mu<c$), in which case it is zero.

We conclude that, since $k \le u$, $pm_c(k) \ge pm_c(u)$ implies that $k=u$, $u\le c$, or $k=u-1$ and $c \mid u$.
\end{proof}

\subsection{Block-Equitably Coloured Packings}
\label{Section-eq_pack}

For packings, all possible choices of $v$, $k$ and $\lambda$ are admissible, and packings exist for any size (number of blocks) up to some maximum. Hence the principal interest is to determine this maximum size  of the \PD$(v,k,\lambda)$.
Here we study an extension of this problem and determine, given positive integers $v \geq k \geq 3$ and $c\geq 2$, the maximum size of a \PD$(v,k,1)$ under the constraint that the packing can be block-equitably $c$-coloured. 

While determining the size of the maximum equitable colourings of packings is largely unexplored, the earlier mentioned paper by Gabrys,  Dau,  Colbourn and  Milenkovic~\cite{GDCM2020} motivates this study through the encoding of data in DNA strands, presenting the theory in terms of balanced set codes. For $c=2$, results have been obtained in work on {\em balanced set codes} with small intersections \cite{GDCM2020, Yu2023}, where the requirement is that each subset of size $t$ occurs in at most one block, where $t\geq 2$ is a specified parameter; the case $t=2$ is relevant to the current discussion.  
Specifically, Lemma~1  and Corollary~1 of   \cite{GDCM2020} 
 provide a  bound on size that is equivalent to our Theorem \ref{bound with general m} in the case where $t=2$ and $c=2$; our bound 
generalises these results in~\cite{GDCM2020} to the case $c>2$ when $t=2$.
The bound given in~\cite{GDCM2020}
is tight in the case $k=3$, $c=2$, $t=2$:

\begin{theorem}[Theorem 3 of~\cite{GDCM2020}; see also Theorem 3.3 of~\cite{Yu2023}]
\label{2_v_3}
For any $v \geq 3$ except for $v\in\{4,5\}$, 
there exists a \PD$(v,3,1)$ of size $m$ with a block-equitable $2$-colouring if and only if
$$m \leq \left\lfloor \frac{1}{2} \left\lfloor\frac{v}{2}\right\rfloor \left\lceil\frac{v}{2}\right\rceil \right\rfloor 
=\left\lfloor\frac{v^2}{8}\right\rfloor.
$$
\end{theorem}

Gabrys, Dau, Colbourn and Milenkovic~\cite{GDCM2020} also provide partial results for $k=t+1$, $t\geq 3$. 
Here we keep the focus on $t=2$, and extend the bound to cover $c>2$ and any $k\geq 3$:
\begin{theorem}
\label{bound with general m}
Let $k=qc+b$  
and $v = \rho c+\beta$,  
where $q$ and $\rho$ are integers,  $0\leq b < c$ and $0\leq \beta < c$.
If there exists a block-equitably $c$-coloured  \PD$(v,k,1)$ of size $m$, then
\[
m\le \frac{\rho\left(\frac{\rho c}{2}+\beta\right)(c-1) + \frac{\beta}{2}(\beta-1)}{q\left(\frac{qc}{2}+b\right)(c-1) + \frac{b}{2}(b-1)}.
\]
\end{theorem}
\begin{proof} Suppose that $v_i$ points are assigned colour $i$, for each $i \in [c]$. The number of non-monochrome pairs in $V$ is $\sum_{i<j} v_iv_j$. Each pair is contained in at most one block, and the number of non-monochrome pairs in any block is $nm_c(k)$.  Thus
\begin{eqnarray*}
m&\le\displaystyle{  \frac{\sum_{i<j} v_iv_j}{nm_c(k)}}.\\
\end{eqnarray*}
 As a consequence of Lemma~\ref{lm:Cauchy}, $\sum_{i<j} v_iv_j \leq nm_c(v)$, with equality when $v_i \in \{ \lfloor{\frac{v}{c}}\rfloor, \lceil{\frac{v}{c}}\rceil \}$ for each $i$.  Noting that $\rho=\left\lfloor \frac{v}{c}\right\rfloor$, and referring to the equations preceding Lemma~\ref{lm:Cauchy} for $nm_c(\mu)$, establishes the result.
\end{proof}

Transversal designs can be used to give a class of packings that meets this bound:

\begin{theorem}\label{lm:tdpacking}
Suppose a $\TD(k,g)$ exists, and let $|V|=v=kg$. If $c \mid k$, then there exists a block-equitably $c$-coloured \PD$(v,k,1)$ that meets the bound given in Theorem~\ref{bound with general m}.
\end{theorem}

\begin{proof}  Since $c \mid k$, there exists a non-negative integer $a$ such that $k=ac$.  In this case $v=agc$ and the bound of Theorem~\ref{bound with general m} simplifies to
\begin{eqnarray*}m&\le& 
 \frac{\left\lfloor\frac{ag c}{c}\right\rfloor
\left(\left\lfloor\frac{ag  c}{c}\right\rfloor\frac{c}{2}+0\right)(c-1) + \frac{0}{2}(0-1)}
{\left\lfloor\frac{ac}{c}\right\rfloor\left(\left\lfloor\frac{ac}{c}\right\rfloor\frac{c}{2}+0\right)(c-1) + \frac{0}{2}(0-1)}\\
&=&
\frac{a^2g^2}
{a^2}=g^2.
\end{eqnarray*}

The TD$(k,g)$ is a packing which satisfies this bound with equality, and a block-equitable $c$-colouring can be obtained by applying each colour class to the union of $k/c$ groups.
\end{proof}

As we previously noted in the introduction, it is known that for a fixed value of $k$, there exists $n_k$ such that a TD$(k,g)$ exists whenever $g \geq n_k$, see~\cite[Section III.3]{Handbook}.  In particular, this result implies that for fixed $k$ and $c \mid k$, there is a block-equitably $c$-colourable $\PD(kg,k,1)$ meeting the bound of Theorem~\ref{bound with general m} for all $g\geq n_k$.

The bound given in Theorem~\ref{bound with general m} can be improved when $c \mid k$ more generally and also for specific values of $c$ and $k$. 
The case when $c=2$ and $k=4$ was considered in \cite{Yu2023}, where the following improved bound was derived.

\begin{lemma}[Proposition 3.1 of~\cite{Yu2023}]
\label{lem:2_4_bound}
Suppose that there exists a block-equitable $2$-colouring of a  \PD$(v,4,1)$ of size $m$. Then $m$ must satisfy 
   \[m \leq \begin{cases}
        n^2, & \text{for } v=4n,\,4n+1,\\
        n^2 + n, & \text{for } v=4n+2,\,4n+3.
    \end{cases}\label{eq:k=4-bound}
    \]
\end{lemma}

In \cite{Yu2023} the authors also proved that this bound is met except for four small cases.

\begin{theorem}[Theorem 3.2 of~\cite{Yu2023}]
\label{2_v_4}
For any $v \geq 4$ except for $v\in\{6,8,9,10\}$, there exists a  \PD$(v,4,1)$ of size $m$ with a block-equitable $2$-colouring if and only if $m$ meets the bound given in Lemma~\ref{eq:k=4-bound}.
\end{theorem}

For the exceptions, $v = 6,8,9,10$, packings meeting the bound must have size $2,4,4,6$ respectively. It is easily seen that no such packings exist, regardless of colouring. 
For example, in the case $v=9$ the pigeonhole principle implies that there exists at least one point in two blocks, so we can assume blocks $\{1,2,3,4\}, \{1,5,6,7\}$ without loss of generality. The remaining two blocks may each contain at most one point from each of these blocks, and hence must contain each of the remaining points, labelled $8$ and $9$. It follows that the pair $\{8,9\}$ is repeated, giving a contradiction.

The construction from~\cite{Yu2023} for the case $c=2$ and $k=4$ relies on Howell designs. When $v \equiv 0, 1 \pmod{4}$, the Howell designs used are equivalent to transversal designs.  However, the cases where $v\equiv 2, 3 \pmod 4$ rely on recursively constructed Howell designs. We give a direct construction for these cases below. This is particularly important as the
use of balanced set codes in DNA data encoding can benefit from 
explicit and efficient computation of designs with specified parameters~\cite{GDCM2020}.
In the remainder of this section we provide an alternative construction when $c=2$ and $k=4$; we include a construction equivalent to~\cite{Yu2023} in the cases $v \equiv 0, 1 \pmod{4}$ for completeness.

\begin{lemma}
\label{2_4_constr_0m4}
For any $n\ge 1$ except $n=2,6$, there exist a \PD$(4n,4,1)$ and a \PD$(4n+1,4,1)$, both of size $n^2$, that each have a block-equitable $2$-colouring.  Further, if there exists a \PD$(4n+2,4,1)$ of size $n^2+n$ with a block-equitable $2$-colouring, then there exists a \PD$(4n+3,4,1)$ of size $n^2+n$ with a block-equitable $2$-colouring.
\end{lemma}

\begin{proof}
Since a TD$(4,n)$ exists for all $n\ge 1$ except $n=2,6$ (see~\cite{Handbook}), Theorem \ref{lm:tdpacking} settles existence for the \PD$(4n,4,1)$. Adding a point of any colour class will give a \PD$(4n+1,4,1)$ of the same size. The definition of a packing does not require that every point occurs in a block, although this could be achieved by adding the new point to a single block, and removing a point of the same colour. The existence of the \PD$(4n+3,4,1)$ follows from the \PD$(4n+2,4,1)$ in the same way.
\end{proof}

It remains only to construct a block-equitable $2$-colouring of a \PD$(4n+2,4,1)$ of size $n^2+n$.

\begin{lemma}
\label{2_4_constr_6m8}
For any odd $n\ge 3$, there exists a \PD$(4n+2,4,1)$  of size $n^2+n$ with a block-equitable 2-colouring.
\end{lemma}
\begin{proof}
Let $n=2s+1$, $v=4n+2=8s+6$.
Let $V_{1}=\{0,1,\dots,2s-1\}\times \{1\} \cup
\{0,1,\dots,2s+1\}\times \{2\}$ and
$V_{2}=\{0,1,\dots,2s+1\}\times \{3,4\}$ be the points of colour $1$ and $2$ respectively, $V=V_{1}\cup V_{2}$.
Taking all addition modulo $2s+2$, let
\begin{eqnarray*}
  {\cal B} &=& \{ \{(i,1),(j,2),(i+j+3,3),(s+2i+j+5+I(i\ge s),4)\} \mid  i \in \mathbb{Z}_{2s},  j \in \mathbb{Z}_{2s+2} \} \\
  &\cup & \{ \{(j,2),(j-1,2),(j+1,3),(s+2+j,4)\} \mid  j \in \mathbb{Z}_{2s+2} \}.
\end{eqnarray*}
Here $I$ $(i\ge s)$ is the indicator function, equal to 1 when $i\ge s$ and 0 otherwise.
Then $(V,{\cal B})$ is the required packing. 
\end{proof}

\begin{lemma}
\label{construct_from_pairs}
Given any $s \ge 2$, let $t\in [2s-1]\setminus\{s\}$, $d=\min(2t,4s-2t)$, and $a_j,b_j\in [4s-1]\setminus\{2s\}$ for $j\in[s-1]$, such that the following all hold:
\begin{itemize}
    \item $a_j<b_j$, for $j\in[s-1]$.
    \item The elements of $\{\min(a_j,4s-a_j) \mid j\in [s-1]\} \cup \{\min(b_j,4s-b_j) \mid j\in [s-1]\} \cup \{t\}$ are all distinct (equivalently, this set equals $[2s-1]$).
    \item The elements of $\{\min(b_j-a_j,4s-b_j+a_j) \mid j \in [s-1]\} \cup \{2s,d\}$ are all distinct.
    \item The elements of $\{ \min(a_j \oplus b_j,4s-(a_j \oplus b_j)) \mid j \in [s-1]\} \cup \{2s\}$ are all distinct, where $\oplus$ denotes addition modulo $4s$.
    \item $4s / \gcd(d,4s)$ is even.
\end{itemize}
Then there exists a \PD$(8s+2,4,1)$ of size $4s^2+2s$ with a block-equitable 2-colouring.
\end{lemma}
\begin{proof}
Take a point set $V=\{\infty_0,\infty_1\} \cup \left( \Z_{4s} \times \{1,2\}\right)$, with colour classes $V_{1}=\{\infty_0,\infty_1\} \cup \left( \Z_{4s} \times \{1\}\right)$ and $V_{2}=\Z_{4s} \times \{2\}$.
Define a function $h:\Z_{4s} \to \{0,1\}$ as follows. Let $g=\gcd(d,4s)$, and note that any $p \in \Z_{4s}$ can be uniquely written as $p=\alpha+\beta d$ (mod $4s$) where $\alpha\in\{0,1,\dots ,g-1\}$ and $\beta\in\{0,1,\dots,4s/g-1\}$. Then define $h(p)=\beta$ (mod 2). Note that $h(p) \ne h(p+d)$ for any $p\in \Z_{4s}$. To understand this function, it may be useful to consider the graph with vertices $\Z_{4s}$ and edges induced by the difference $d$. This graph consists of $g$ $4s/g-$cycles. We assign the vertices to two classes such that adjacent vertices have different classes; this is possible if and only if the cycle length is even. We now define the blocks
\begin{eqnarray*}
    {\cal B} &=& \{ \{ (p,1),(p+a_i+b_i,1),(p+a_i,2),(p+b_i,2)\} \mid p \in \Z_{4s}, i \in [s-1]\} \\
    &\cup& \{\{(p,1),(p+2s,1),(p,2),(p+2s,2)\} \mid p \in \{0,1,\dots,2s-1 \} \} \\
    &\cup& \{\{\infty_{h(p)},(p,1),(p-t,2),(p+t,2)\} \mid p \in \Z_{4s} \}.
\end{eqnarray*}
Then $(V,{\cal B})$ is the required packing. 
\end{proof}

Note that we may restrict $a_i,b_i\in [2s-1]$, in which case the statement of these conditions is simplified.

\begin{lemma}
\label{2_4_constr_2m8}
For any $s\ge 2$, there exist $t$ and $a_j, b_j$, $j\in [s-1]$, which satisfy the conditions of Lemma
\ref{construct_from_pairs}.
\end{lemma}
\begin{proof}
The general construction is given in eight cases depending on $s$ modulo 12. In each case we give three sets of $(a,b)$ pairs, each defined for each $i$ in a specified range ${\cal S}$, as well as $t$ and up to three additional $(a,b)$ pairs.

\vspace{2mm}
$\begin{array}{|ll|l|ll|ll|} \hline
    \multicolumn{2}{|l|}{s \pmod{12}}  & t & {\cal S} & (a,b) & {\cal S} & (a,b)  \\ \hline
    2,10 & s\geq10 & s-1 & [\frac{s}{2}] & (i,\frac{3s}{2}+1-2i) & [\frac{s-6}{4}] & (s-1-2i,\frac{3s}{2}+i) \\
    0,8 & s\geq8 & s-3 & [\frac{s}{2}] & (i,\frac{3s}{2}+2-2i) & [\frac{s}{4}]-\{2\} & (s+1-2i,\frac{3s}{2}-2+i) \\
    4 & s\geq16 & s-1 & [\frac{s}{2}] & (i,\frac{3s}{2}+2-2i) & [\frac{s-16}{4}] & (s-3-2i,\frac{3s}{2}+2+i) \\
    6 & s\geq18 & s+1 & [\frac{s-2}{2}] & (i,\frac{3s}{2}-1-2i) & [\frac{s+2}{4}] & (s+1-2i,\frac{3s}{2}-2+i) \\
    1,9 & s\geq13 & s-4 & [\frac{s-1}{2}] & (i,\frac{3(s-1)}{2}+2-2i) & [\frac{s-5}{4}]-\{2\} & (s-2i,\frac{3(s-1)}{2}+1+i) \\
    3,11 & s\geq15 & s+6 & [\frac{s-1}{2}] & (i,\frac{3(s-1)}{2}+1-2i) & [\frac{s+1}{4}] & (s+2-2i,\frac{3(s-1)}{2}+1+i) \\
    5 & s\geq17 & \frac{3(s-1)}{2}+1 & [\frac{s-1}{2}] & (i,\frac{3(s-1)}{2}+1-2i) & [\frac{s-1}{4}]-\{2\} & (s+1-2i,\frac{3(s-1)}{2}+2+i) \\
    7 & s\geq19 & \frac{3(s-1)}{2}+2 & [\frac{s-1}{2}] & (i,\frac{3(s-1)}{2}+2-2i) & [\frac{s+1}{4}]-\{2\} & (s+1-2i,\frac{3(s-1)}{2}+i) \\ \hline
\end{array}$

\vspace{2mm}
$\begin{array}{|l|ll|l|} \hline
    s \pmod{12} & {\cal S} & (a,b)  & \text{Extra $(a,b)$ pairs} \\ \hline
    2,10 & [\frac{s+2}{4}] & (s-1+2i,2s-i) &  \\
    0,8 & [\frac{s}{4}-1] & (s-1+2i,2s-i) & (2s-1-\frac{s}{4},2s-\frac{s}{4}) \\
    4 & [\frac{s}{4}] & (s+1+2i,2s-i) & (s-3,s+1),(\frac{s}{2}+1,\frac{7s}{4}-1),(\frac{s}{2}+3,\frac{3s}{2}+2) \\
    6 & [\frac{s-6}{4}] & (s+1+2i,2s-1-i) & (\frac{7s-2}{4},2s-1) \\
    1,9 & [\frac{s+3}{4}] & (s+2i,2s-i) & (\frac{s-1}{2}+1,s) \\
    3,11 & [\frac{s-3}{4}]-\{3\} & (s+2i,2s-i) & (\frac{3(s-1)}{2}+1,2s-3) \\
    5 & [\frac{s-5}{4}] & (s+5+2i,2s-i) & (s-3,s+5),(s+1,s+3) \\
    7 & [\frac{s-7}{4}] & (s+1+2i,2s-i) & (s-3,s+1),(\frac{7s-1}{4},\frac{7s+3}{4}) \\ \hline
\end{array}$
\smallskip

In each case, it is straightforward (if tedious) to verify that the five conditions of Lemma~\ref{construct_from_pairs} hold.
Thus it only remains to verify the existence of a few small cases, and these are dealt with in the following table.

\vspace{2mm}
$\begin{array}{|l|l|l|} \hline
s & t & \text{$(a,b)$ pairs} \\ \hline
2 & 1 & (2, 3) \\
3 & 5 & (1, 4), (2, 9) \\
4 & 7 & (1, 5), (2, 12), (3, 6) \\
5 & 9 & (1, 5), (2, 14), (3, 8), (4, 13) \\
6 & 1 & (2, 6), (3, 11), (4, 7), (5, 10), (8, 9) \\
7 & 1 & (2, 3), (4, 9), (5, 11), (6, 13), (7, 10), (8, 12) \\
9 & 7 & (2, 3), (4, 9), (1, 5), (6, 14), (8, 17), (10, 16), (11, 13), (12, 15) \\
11 & 13 & (2, 3), (4, 9), (1, 10), (5, 15), (6, 17), (7, 19), (8, 11), (12, 18), (14, 21), (16, 20) \\ \hline
\end{array}$

\end{proof}

The constructions in Lemmas \ref{2_4_constr_0m4}, \ref{2_4_constr_6m8}, \ref{construct_from_pairs} and \ref{2_4_constr_2m8} prove the sufficiency of the bound \eqref{eq:k=4-bound} from Lemma~\ref{lem:2_4_bound}, except for $v \in \{2,3,6,7,8,9,10,11,24,25\}$. The trivial packing with no blocks satisfies $v=2,3$, and we have observed that the bound cannot be met for $v = 6,8,9,10$. This leaves only $v=7,11,24,25$.

For $v=7$, the necessary construction is given by colour classes $\{a_1,a_2,a_3\}$, $\{b_1,b_2,b_3,b_4\}$ and blocks $\{a_1, a_2, b_1, b_2\}$, $\{a_1, a_3, b_3, b_4\}$. 
For $v=11$, the necessary construction is given by colour classes $\{a_1,a_2,a_3,a_4,a_5\}$, $\{b_1,b_2,b_3,b_4,b_5,b_6\}$ and blocks 
$\{a_1, a_3, b_1, b_4\}$, 
$\{a_1, a_4, b_2, b_5\}$, 
$\{a_1, a_5, b_3, b_6\}$,
$\{a_2, a_3, b_2, b_6\}$,
$\{a_2, a_4, b_3, b_4\}$,
$\{a_2, a_5, b_1, b_5\}$.
For $v=24$, the necessary construction is given by colour classes $\mathbb{Z}_{12} \times \{1\}$, $\mathbb{Z}_{12} \times \{2\}$ and block set 
$\{ 
\{(i,1),(i+5,1),(i+2,2),(i+4,2)\},
\{(i,1),(i+3,1),(i-2,2),(i+6,2)\},
\{(i,1),(i+4,1),(i,2),(i+5,2)\} 
\mid i \in \mathbb{Z}_{12} \}$, and the necessary $v=25$ construction follows as in Lemma \ref{2_4_constr_0m4}.

With these final cases, we have reproduced the findings of Theorem~\ref{2_v_4} using explicit constructions. 

\section{Weak Colouring}\label{sc:weak}

In this section we consider weak colourings of BIBDs and GDDs. Recall that for a block design $\cal D$, the chromatic number $\chi({\cal D})$ is the smallest number of colours needed to weakly colour $\cal D$.
Since the existence of a block-equitable $c$-colouring of a block design ${\cal D}$ implies that $\chi({\cal D})\leq c$, results in Section~\ref{sc:equitable} can be used to obtain upper bounds on the chromatic number of specific designs. 
However, weak colouring is a significantly more permissive condition, with these bounds not being tight in most cases.

After briefly reviewing weak colourings of BIBDs, we focus on weak colourings of GDDs. 

\subsection{Weakly Coloured BIBDs}\label{ssc:weakBIBD}

Rosa and Colbourn~\cite{RosaColbourn1992} surveyed results on 
colourings of block designs. The relevant results for weak colourings are discussed below.

For $k=3$, Rosa~\cite{Rosa1970} and independently Pelik\'{a}n~\cite{Pelikan1970} proved that for every STS$(v)$ ${\cal D}$ with $v>3$, $\chi ({\cal D})> 2$. This also follows from Theorem~\ref{thm:equitableBIBD} due to the equivalence between weak and block-equitable colouring when $c=2$ and $k=3$.
For each integer $c \geq 3$ there is a $c$-chromatic STS$(v)$ of some order~\cite{deBPR1982,Rosa1970}, although
for each admissible order $7 \leq v \leq 19$ every STS$(v)$ is 3-chromatic~\cite{STS19,MPR1983}.
It has also been established that
there exists a 2-chromatic BIBD$(v,4,\lambda)$ for each admissible order, see~\cite{FGLR2002,HLP1990,HLP1991,RosaColbourn1992}
and that there exists a 2-chromatic BIBD$(v,5,1)$ for each admissible order, see~\cite{Ling1999}.

Horsley and Pike
established the asymptotic existence of $c$-chromatic BIBDs, with the following theorem.
\begin{theorem}[Theorem 1.1 of~\cite{horsley_pike}]
\label{Thm:HP-Thm1.1}
For all integers $c \geq 2$, $k \geq 3$, $\lambda \geq 1$ such that $(c,k) \neq (2,3)$, there is an integer $N(c,k,\lambda)$ such that for all
$(k,\lambda)$-admissible integers $v \geq N(c,k,\lambda)$, there exists a $c$-chromatic BIBD$(v,k,\lambda)$.
\end{theorem}

\subsection{Weakly Coloured GDDs}
\label{ssc:weakGDD}

The results in this section mostly concern uniform $k$-GDD$_{\lambda}$. 
We note that in an unpublished manuscript, Li and Shen~\cite{LiShen} remarked that $c$-chromatic GDDs coupled with transversal designs could be used to produce $c$-chromatic GDDs with expanded group sizes. 
We extend these results in the present paper and use the technique to establish existence results for $c$-chromatic  uniform $k$-GDDs. 

Recalling that when $k=3$ and $c=2$, weak and block-equitable colouring are equivalent, we obtain the following corollary of Theorem~\ref{th:equitableGDD}.
\begin{corollary}
\label{Cor-2chr3GDD}
If there exists a $2$-chromatic $3$-GDD of type $w^u$, then $u \in \{3,4\}$.
\end{corollary}

We now present a general bound that allows for multiple block sizes.

\begin{theorem}
\label{thm:bounds_chrom}
If ${\cal D}$ is a $K$-GDD$_{\lambda}$ of type $w^u$, then for
$k_{\min} = \min\{k : k \in K\}$
\begin{eqnarray*}
\chi({\cal D}) \leq \left\lceil \frac{u}{k_{\min}-1}\right\rceil.
\end{eqnarray*}
Moreover, if $K=\{k\}$ and $k=u$, then $\chi(D)=2$.
\end{theorem}
\begin{proof}

Partition the groups into $\left\lceil \frac{u}{k_{\min}-1}\right\rceil$ sets, $S_i$, each containing at most $k_{\min}-1$ groups.
Assign colour $i$ to each point in each group of the set $S_i$.
Since the blocks are of size $k_{\min}$ or greater, all blocks must contain at least two points of different colours.
Further, if each block intersects every group two colours will suffice.
\end{proof}

We note that the case $K=\{k\}$ and $k=u$ corresponds with a TD$(k,u)$. Moreover, this colouring can be made block-equitable by Lemma~\ref{lm:trivialGDD}.

In general, the upper bound of Theorem~\ref{thm:bounds_chrom} is not tight; indeed  Corollary \ref{blow-up_corollary1} establishes the existence of infinite classes of GDDs that do not meet this bound. 
To see this we apply a well-known construction for general GDDs, taking care to ensure that the GDD's chromatic number does not exceed that of the original design.

\begin{theorem} \label{blow-up_colouring}
If there exist a $c$-chromatic $k$-GDD$_{\lambda}$, ${\cal D} = (V, {\cal G}, {\cal B})$,  with  ${\cal G}=\{G_1,\dots, G_u\}$, and a  TD$(k,w)$, then there exists a $c$-chromatic $k$-GDD$_{\lambda}$, ${\cal D}^\times=(V^\times, {\cal G}^\times, {\cal B}^\times)$, with ${\cal G}^{\times}=\{G_1^{\times},\dots, G_u^{\times}\}$ where $|G_i^{\times}|=w|G_i|$.
\end{theorem}

\begin{proof}
Suppose ${\cal D} = (V, {\cal G}, {\cal B})$  is a $k$-GDD$_{\lambda}$ with  ${\cal G}=\{G_1,\dots, G_u\}$, and ${\cal T}$ is a TD$(k,w)$. Construct a  $k$-GDD$_{\lambda}$ ${\cal D}^{\times}=(V^{\times}, {\cal G}^{\times}, {\cal B}^{\times})$  with ${\cal G}^{\times}=\{G_1^{\times},\dots, G_u^{\times}\}$, where $V^{\times}=V\times \{0,1,\dots,w-1\}$ and $G_j^{\times}=G_j\times \{0,1,\dots,w-1\}$, so $|G_i^{\times}|=w|G_i|$. The blocks of ${\cal B}^\times$ are obtained by placing a copy of the TD$(k,w)$ on the groups $\{x_i \times\{0,1,\dots,w-1\} \mid x_i\in B \}$ for each $B\in {\cal B}$.  Without loss of generality, by permuting the points within each group if necessary, we may assume that one of the blocks of the TD$(k,w)$ is $B\times \{0\}$.
 
Extend a $c$-colouring $\phi$ of ${\cal D}$ to a $c$-colouring $\psi$ of ${\cal D}^\times$  by setting $\psi((x,i)) = \phi(x)$ for each $x\in V$ and $i \in\{0,1,\ldots, w\}$. Since each block of ${\cal D}$ was expanded to a TD$(k,w)$ containing the block $B\times\{0\}$, ${\cal D}^\times$ contains at least one subdesign isomorphic to ${\cal D}$. By assumption ${\cal D}$ is $c$-chromatic and so the resultant ${\cal D}^\times$ cannot be coloured with fewer than $c$ colours.
\end{proof}

For asymptotic behaviour, we note that a $\bibd(v,k,\lambda)$ is a $k$-GDD$_{\lambda}$ of type $1^v$; hence Theorem~\ref{Thm:HP-Thm1.1} establishes 
the existence of $c$-chromatic  $k$-GDD$_{\lambda}$s of type $1^v$.  Combined with Theorem~\ref{blow-up_colouring}, we obtain the following.

\begin{corollary} \label{blow-up_colouring_BIBD}
If there exist a $c$-chromatic $\bibd(v,k,\lambda)$ and a TD$(k,w)$, then there exists a $c$-chromatic $k$-GDD$_{\lambda}$ of type $w^v$.
\end{corollary}

\begin{example}
Consider the $\bibd(7,3,1)$, $\mathcal{D}$, on point set $\mathbb{Z}_7$ whose blocks are listed in Example~\ref{Ex:STS(7)}.  
It is well-known (and easy to verify) that $\chi(\mathcal{D})=3$.  A $3$-colouring with colour classes 
$C_1=\{{\color{blue}0},{\color{blue}1},{\color{blue}2}\}$, $C_2 = \{{\color{red}\mathit{3}},{\color{red}\mathit{4}}\}$ and $C_3=\{{\color{darkgreen}\mathbf{5}},{\color{darkgreen}\mathbf{6}}\}$ 
was illustrated in Example~\ref{Ex:STS(7)}.  

Consider the following TD$(3,2)$ on point set $\{x,y,z\} \times \{0,1\}$:
\[
\{(x,0),(y,0),(z,0)\}, \{(x,0), (y,1), (z,1)\}, \{(x,1),(y,1),(z,0)\}, \{(x,1),(y,0),(z,1)\}.
\]
For each block $B\in\mathcal{D}$, we replace $\{x,y,z\}$ with the points of $B$ to obtain a TD$(3,2)$ on point set $B \times \{0,1\}$.
This gives a $3$-chromatic $3$-GDD of type $2^7$ on point set $\{0,\ldots,6\} \times \{0,1\}$.  An explicit $3$-colouring can be obtained by colouring the elements of $C_i \times \{0,1\}$ with colour $i$.

To illustrate the colouring, consider the block $\{{\color{blue}0},{\color{blue}1},{\color{red}\mathit{3}}\}$ of $\mathcal{D}$; the corresponding blocks of the GDD are
\[
\{{\color{blue}(0,0)},{\color{blue}(1,0)},{\color{red}\mathit{(3,0)}}\}, \{{\color{blue}(0,0)}, {\color{blue}(1,1)}, {\color{red}\mathit{(3,1)}}\}, \{{\color{blue}(0,1)},{\color{blue}(1,1)},{\color{red}\mathit{(3,0)}}\}, \{{\color{blue}(0,1)},{\color{blue}(1,0)},{\color{red}\mathit{(3,1)}}\}.
\]
Since $0,1 \in C_1$ and $3 \in C_2$, it follows that each of these four blocks has two points of colour $1$ and one of colour $2$.
\end{example}

\noindent
More generally, we have the following.

\begin{corollary}\label{Cor:DianeHomework}
Except when $(c,k) = (2,3)$, for each $c \geq 2$ and each $k$ such that $k-1 = q^\alpha$ where $q$ is prime and $\alpha$ is a positive integer,
there exists a $c$-chromatic $(k-1)$-uniform $k$-GDD.
\end{corollary}

\begin{proof}
The case $(c,k)=(2,3)$ is an exception by Corollary~\ref{Cor-2chr3GDD}. Otherwise, a $c$-chromatic BIBD$(v,k,1)$ exists by Theorem~\ref{Thm:HP-Thm1.1}, and
a TD$(k,k-1)$ exists whenever $k-1$ is a prime power.
\end{proof}

\begin{corollary} \label{blow-up_corollary1}
Let $k\geq 3$, $\lambda \geq 1$ and  $c \geq 2$ be integers with $(c,k) \neq (2,3)$.  If there exists a $\TD(k,w)$, then there is an integer $u_0$ such that for all $(k,\lambda)$-admissible $u \geq u_0$, there exists a $c$-chromatic $k$-GDD$_{\lambda}$ of type $w^u$.
\end{corollary}

\begin{proof}
By Theorem~\ref{Thm:HP-Thm1.1}, there is an integer $u_0$ such that whenever $u \geq u_0$ and $u$ is $(k,\lambda)$-admissible, there exists a $c$-chromatic $\bibd(u,k,\lambda)$.  Using this as an ingredient in Corollary~\ref{blow-up_colouring_BIBD}, the result follows.
\end{proof}

\begin{corollary} \label{blowing-up_corollary2}
Let $k \geq 3$, $\lambda \geq 1$ and $c \geq 2$ be integers with $(c,k) \neq (2,3)$.  There exist integers $u_0, w_0$ such that for every integer $w \geq w_0$ and every $(k,\lambda)$-admissible integer $u \geq u_0$, there is a $c$-chromatic $k$-GDD$_{\lambda}$ of type $w^u$.
\end{corollary}

\begin{proof}
Given $k$, there exists an integer $w_0$ such that for every $w \geq w_0$, there exists a $\TD(k,w)$ (see~\cite{Mohacsy2011} and references therein).
The result now follows readily from Corollary~\ref{blow-up_corollary1}.
\end{proof}

Corollary~\ref{blowing-up_corollary2} establishes the asymptotic existence of uniform $k$-GDDs with arbitrarily many groups and arbitrary chromatic number
(with the exception of the case $c=2$  and $k=3$ by Corollary~\ref{Cor-2chr3GDD}). 

If we specifically focus on GDDs derived from converting a parallel class of a BIBD into groups, then the chromatic number of the resultant GDD may be bounded by that of the BIBD, as shown  in  Theorem~\ref{thm:parallel}.

\begin{theorem}\label{thm:parallel}
Let ${\cal D}=(V,{\cal B})$ be a BIBD$(v,k,1)$,  such that ${\cal B}$ contains a parallel class $\pi$, where ${\cal D}$ has chromatic number $\chi({\cal D})$. Then there exists a $k$-GDD  ${\cal D}_{\pi}$ of type $k^{v/k}$ such that 
\begin{eqnarray*}
 \chi({\cal D})\geq \chi({\cal D}_{\pi})\geq \chi({\cal D})-\left\lceil \frac{v}{k(k-1)} \right\rceil.
\end{eqnarray*}
\end{theorem}

\begin{proof} Suppose there exists a  BIBD$(v,k,1)$ ${\cal D}=(V,{\cal B})$ containing a parallel class $\pi$. Construct $k$-GDD ${\cal D}_\pi=(V,{\cal G}_\pi,{\cal B}_\pi)$ of type $k^{v/k}$ where the  $\frac{v}{k}$  blocks of $\pi$ form the groups of ${\cal G}_\pi$ and ${\cal B}_\pi={\cal B}\setminus \pi$. 
It is clear that $\chi({\cal D}_{\pi}) \leq \chi({\cal D})$ so we need only prove that
$\chi({\cal D}) - \lceil\frac{v}{k(k-1)}\rceil \leq \chi({\cal D}_{\pi})$.
Let $\psi$ be a colouring of ${\cal D}_{\pi}$ with $\chi({\cal D}_{\pi})$ colours.
If $\psi$ is not a valid colouring for ${\cal D}$ then when $\psi$ is applied to the blocks of ${\cal D}$ there must be one or more  blocks of $\pi$  that are monochromatic.
Since the blocks of $\pi$ are disjoint  we may identify a subset of $v/k$ distinct points, $\{x_1,x_2,\ldots,x_{v/k}\}$, one from each  block of the parallel class. Form a partition of $\{x_1,x_2,\ldots,x_{v/k}\}$
into $\lceil \frac{v}{k(k-1)} \rceil$ parts, with each part of size at most  $k-1$. Assign a distinct new colour to each part, and colour all points in each part with the assigned colour.
Clearly then
$\chi({\cal D})  \leq \chi({\cal D}_{\pi}) + \lceil \frac{v}{k(k-1)} \rceil$.
\end{proof}

Further results can be obtained by deleting a point of a BIBD of index $\lambda=1$.

\begin{lemma}\label{lm:D-1D}
Suppose ${\cal D}=(V,{\cal B})$ is a BIBD$(v,k,1)$, with chromatic number $\chi({\cal D})$. Then there exists a $k$-GDD of type $(k-1)^{(v-1)/(k-1)}$, denoted ${\cal D}_y$, with chromatic number $\chi({\cal D}_y) \in \{ \chi({\cal D})-1, \chi({\cal D}) \}$.
Further, if there exists a $\chi({\cal D})$-colouring of ${\cal D}$ with a colour class that contains only one point, then there exists such a ${\cal D}_y$ with  $\chi({\cal D}_y) = \chi({\cal D}) -1$.
\end{lemma}

\begin{proof} Suppose ${\cal D}=(V,{\cal B})$ is a BIBD$(v,k,1)$, with $y\in V$ and $Y= \{B\in {\cal B} \mid y\in B\}$. Construct  $k$-GDD ${\cal D}_y=(V_y,{\cal G}_y, {\cal B}_y)$ of type $(k-1)^{\frac{v-1}{k-1}}$ where $V_y=V\setminus\{y\}$, ${\cal G}_y=\{ B\setminus\{y\} \mid B \in Y\}$ and ${\cal B}_y={\cal B}\setminus Y$.
Any colouring of ${\cal D}$ yields a colouring of ${\cal D}_y$ and so $\chi({\cal D}_y) \leq \chi({\cal D})$.

Consider the case $\chi=\chi({\cal D}_y) < \chi({\cal D}) - 1$, and let  $\psi$ be a $\chi$-colouring of ${\cal D}_y$. In ${\cal D}$, take a colour not in $\psi$ and apply it to point $x$, with the remaining points coloured as in $\psi$. This gives a colouring of ${\cal D}$ using $\chi({\cal D}_y)+1$ colours.  Since $\chi({\cal D}_y)+1<\chi({\cal D})$,  we obtain a contradiction.
Consequently, $\chi({\cal D}_y) \geq \chi({\cal D}) - 1$.

Suppose that some colouring of $\cal D$ with $\chi(D)$ colours has a colour class of size 1. 
Let $y$ be the sole point in this colour class. 
Repetition of the above argument implies that   ${\cal D}_y$ can be coloured with at most $\chi({\cal D})-1$ colours.
  Hence if some colouring of $\cal D$ with $\chi(D)$ colours has a colour class of size 1,
then $\chi({\cal D}_y) = \chi({\cal D}) -1$.
\end{proof}

Applying Lemma~\ref{lm:D-1D} in the case where $k=3$ and $\chi({\cal D}) = 3$, we get the following result.

\begin{theorem}
\label{Cor-ThreeChromaticSTS}
If $\cal D$ is a 3-chromatic $\sts(v)$ and $v > 9$ then $\chi({\cal D}_y) = \chi({\cal D}) = 3$ for each point $y$ of $\cal D$.
\end{theorem}

\begin{proof}
By Lemma~\ref{lm:D-1D}, $\chi({\cal D}_y) = 2$ or $3$.
However, ${\cal D}_y$ is a $3$-GDD of type $2^{(v-1)/2}$; hence, since $v> 9$, by Corollary~\ref{Cor-2chr3GDD}, $\chi({\cal D}_y) > 2$.
\end{proof}

Observe that
Theorem~\ref{Cor-ThreeChromaticSTS} is not informative regarding $2$-uniform GDDs constructed from STSs that have chromatic number 4 or more.
To consider how such GDDs might behave, we examine known examples of  $4$-chromatic STSs. 
Haddad~\cite{Haddad1999} documented the blocks of two $4$-chromatic STS$(v)$ having $v \in \{21, 39\}$ and de~Brandes, Phelps and R\"odl~\cite{deBPR1982} documented the blocks of four $4$-chromatic STS$(v)$ having $v \in \{25,27,33,37\}$.
We will denote this set of six $4$-chromatic STS$(v)$ by ${\collection}$ (we discuss some of these designs further in Section~\ref{Sec:Monogroup}; also see the Appendix at the end of this paper).
For each STS ${\cal D}$ in  ${\collection}$ and each point $y$ we analysed the resulting $2$-uniform $3$-GDD,  ${\cal D}_y$.
For orders $v \in \{37,39\}$ each $3$-GDD, ${\cal D}_y$, was found to be $4$-chromatic,
and so each colour class of any $4$-colouring of the STS ${\cal D}$ must have at least two points.
However, for orders $v \in \{21,25,27,33\}$ each of the  $3$-GDDs, ${\cal D}_y$, was found to be $3$-chromatic, indicating that there do exist $4$-chromatic STSs ({i.e.}, $3$-GDDs of type $1^v$)
for which it is possible to find a $4$-colouring exhibiting a colour class of size 1.

We also investigated the STSs in ${\collection}$ with respect to the bounds given in Theorem~\ref{thm:parallel}.
Four of the $4$-chromatic STSs listed in the collection ${\collection}$
 have parallel classes, namely those of order 21, 27, 33 and 39. The STS$(21)$ has 130 distinct parallel classes; we confirmed that converting the blocks of any of these parallel classes into groups results in a $3$-chromatic $3$-GDD.
For the $4$-chromatic STSs of order $v \in \{27,33,39\}$ in ${\collection}$
we similarly analysed the effect of converting a single parallel class into groups.
The STS$(27)$ has 625 parallel classes, and each of the resulting $3$-uniform $3$-GDDs is $3$-chromatic.
The STS$(33)$ has 18146 parallel classes,
15261 of which yield $3$-chromatic $3$-GDDs and 2885 yield $4$-chromatic $3$-GDDs.
For the STS$(39)$, there are 1075152 parallel classes;
only 795 of the corresponding $3$-GDDs are $3$-chromatic while the other 1074357 are $4$-chromatic.

It would be tempting to ask if there exists an STS, ${\cal D}$, with a parallel class $\pi$
such that $\chi({\cal D}_\pi) = 2$.
However, Corollary~\ref{Cor-2chr3GDD}
 leads to a lower bound on $\chi({\cal D}_\pi)$
that rules out this possibility for  any STS of order greater than $9$.
Corollary~\ref{Cor-2chr3GDD} also forbids $2$-chromatic uniform $3$-GDDs when there are more than four groups.
For an example of a $2$-chromatic uniform $3$-GDD with three groups, note that
when $\cal D$ is the unique STS$(9)$ and $\pi$ is one of its parallel classes,
then ${\cal D}_\pi$ is a $2$-chromatic $3$-GDD of type $3^3$.  More generally, any TD$(3,u)$ is $2$-chromatic.

\begin{corollary}
Let $v > 9$.
If ${\cal D}$ is a STS$(v)$ with a parallel class $\pi$
then $\chi({\cal D}_\pi) \geq 3$.
\end{corollary}

\begin{proof}
    Given an STS, $\mathcal{D}$, of order $v > 9$, any parallel class has at least five blocks.  Thus $\mathcal{D}_{\pi}$ has more than four groups, and the result follows by Corollary~\ref{Cor-2chr3GDD}.
\end{proof}

\subsection{Weakly Coloured Packings}
\label{Sec:Weakly Coloured Packings}

We now consider some questions related to colourings of packings. 
Three natural questions arise in this context.
By analogy with the results in Section~\ref{Section-eq_pack}, a natural problem is to determine, for each $v$, $k$, $\lambda$ and $c$, the maximum size of a (weakly) $c$-coloured \PD$(v,k,\lambda)$. 
However, this would represent a departure from the typical emphasis on chromatic number in weak colouring research; that is, we typically seek to identify weakly $c$-coloured designs in which a $(c-1)$-colouring is not possible. Thus it may be desirable to incorporate this condition, and seek the maximum size $c$-chromatic \PD$(v,k,\lambda)$ (if one exists) for each $v$, $k$, $\lambda$ and $c$. 
We suspect that either of the problems posed above is likely to be very challenging. A third and more tractable problem could be to fix $v$, $k$, and $\lambda$, and then determine the values of $\chi$ for which there exists a \PD$(v,k,\lambda)$ with chromatic number $\chi$. 
Weak colourings of BIBDs and GDDs can then be seen as special cases of this problem.

One approach to constructing packings with bounded chromatic number is to remove blocks from a design having a known chromatic number. The following result bounds the resultant change in chromatic number.

\begin{theorem} \label{thm:remove_blocks}
Let ${\cal D}=(V,{\cal B})$ be a PD$(v,k,1)$.
Let  ${\cal \beta}=\{B_1, B_2, \ldots, B_t\} \subset {\cal B}$ be a set of $t$ blocks.
Let ${\cal D}\setminus{\cal \beta}$ be the design obtained from ${\cal D}$ by removing each of the blocks $B_1, B_2, \ldots, B_t$ from ${\cal B}$.
Then $\chi({\cal D}) - \lceil\frac{t}{k-1}\rceil \leq \chi({\cal D}\setminus{\cal \beta}) \leq \chi({\cal D})$.
\end{theorem}

\begin{proof}
It is clear that $\chi({\cal D}\setminus{\cal \beta}) \leq \chi({\cal D})$ so we need only prove that
$\chi({\cal D}) - \lceil\frac{t}{k-1}\rceil \leq \chi({\cal D}\setminus{\cal \beta})$.
Consider a colouring of ${\cal D}\setminus{\cal \beta}$  using $\chi({\cal D}\setminus{\cal \beta})$ colours.
If this colouring fails to also be a valid colouring for ${\cal D}$ then it must be that one or more of the blocks $B_1, B_2, \ldots, B_t$
is monochromatic with respect to this colouring.
For each $i \in \{1,2,\ldots,t\}$, select a point $x_i$ from $B_i$.
Let $t'=|\{x_1, \ldots, x_t\}|$.  Select $\lceil \frac{t'}{k-1} \rceil \leq \lceil \frac{t}{k-1} \rceil$ distinct new colours and
apply each one
to at most $k-1$ of the points of $\{x_1,x_2,\ldots,x_t\}$, forming a partition of
$\{x_1,x_2,\ldots,x_t\}$ into $\lceil \frac{t'}{k-1} \rceil$ parts of size at most $k-1$.
Clearly then
$\chi({\cal D})  \leq \chi({\cal D}\setminus{\cal \beta}) + \lceil \frac{t'}{k-1} \rceil \leq \chi({\cal D}\setminus{\cal \beta}) + \lceil \frac{t}{k-1} \rceil$.
\end{proof}

Recalling that there exist $c$-chromatic BIBD$(v,k,\lambda)$ for $(c,k) \neq (2,3)$~\cite{horsley_pike}, it would be natural to take $\mathcal{D}$ to be a BIBD in Theorem~\ref{thm:remove_blocks}.
Further, if $\beta$ in Theorem~\ref{thm:remove_blocks} is a parallel class, then necessarily $t'=t$.  In this case, then we have the following corollary, analogous to Theorem~\ref{thm:parallel}.
		
\begin{corollary}
If ${\cal D}$ is a \PD$(v,k,1)$ with a parallel class $\pi$,
then ${\chi({\cal D})} - {\chi({\cal D}\setminus\pi)} \leq \lceil \frac{v}{k(k-1)} \rceil$.
\end{corollary}

\section{Some Other Types of Colourings}\label{OtherColor}

	In this section we briefly consider some other types of colourings that are suggested by our work.
    We first consider colourings of GDDs where the colouring on the groups has a specific form.  
    In Section~\ref{Sec:Monogroup} we consider colouring GDDs requiring the groups to be monochromatic. We note that many of the colourings of GDDs from Section~\ref{ssc:weakGDD} in fact have this property.
    In contrast, Section~\ref{Sec:GroupEquitable} considers the case where we require the groups to be {\em group-equitably} coloured, that is, each group has, as closely as possible, an equal number of points of each colour.

\subsection{Monochromatic Groups} \label{Sec:Monogroup}

Many of the weak colourings that we demonstrated for GDDs in Section~\ref{ssc:weakGDD} happen to have the property that each group of the GDD is monochromatic; for instance, the GDDs constructed 
by Theorem~\ref{thm:bounds_chrom}, as well as Theorem~\ref{blow-up_colouring} and its corollaries, have monochromatic groups.
Because monochromatic groups are not an intrinsic requirement for colourings of GDDs, we define the {\em monochromatic group chromatic number} $\chi_M ({\cal D})$ to be the least number of colours for which there exists a colouring of the GDD, $\cal D$, in which no block is monochromatic but each group is monochromatic.
Clearly $\chi ({\cal D}) \leq \chi_M ({\cal D})$ for every GDD, ${\cal D}$.
Noting that the proof of Theorem~\ref{thm:bounds_chrom} produces a monochromatic group colouring, we immediately have an upper bound on $\chi_M ({\cal D})$.

\begin{corollary}
	If ${\cal D}$ is a $k$-GDD$_\lambda$ of type $g^u$ then $\chi_M({\cal D}) \leq \lceil \frac{u}{k-1} \rceil$.
\end{corollary}

Given a Steiner triple system with a parallel class $\pi$, a natural way to construct a 3-GDD of type $3^u$ is to identify the blocks of the parallel class as the groups.
In 1999, Haddad~\cite{Haddad1999} presented a 4-chromatic STS(21), an isomorphic copy of which we illustrate in Table~\ref{Table:STS21}.
As already noted in Subsection~\ref{ssc:weakGDD},
this particular STS,
which we now denote by $\cal H$, has 130 distinct parallel classes $\pi$, each of which we confirmed produces a 3-chromatic GDD, ${\cal H}_\pi$.
However, for several of these GDDs three colours are insufficient when trying to also ensure that each group is monochromatic.

As a specific example where $\chi({\cal H}_\pi) < \chi_M({\cal H}_\pi)$, let ${\pi}$ be the parallel class consisting of the seven blocks displayed in the top row of Table~\ref{Table:STS21}.
As a counting matter, there are $\binom{7}{3} = 35$ 3-subsets that can be selected from a 7-set, and in our context we have a 7-set $\cal G$ consisting of seven groups (each of which consists of three points).
It is straightforward but tedious to observe that for each of the 35 3-subsets $\cal T$ of the group set $\cal G$ there is at least one block of the GDD ${\cal H}_{\pi}$ that contains one point from each of the three groups of $\cal T$.
If it is possible to colour ${\cal H}_{\pi}$ with three colours so that each group is monochromatic, then necessarily three of the seven groups must share the same colour, which would then result in the contradictory existence of at least one monochromatic block (since at least one block has a point in each of the three like-coloured groups). 
We can obtain a 4-colouring by assigning 3 colours to two groups each
and the last colour to the last group, from which we conclude that $\chi_M({\cal H}_{\pi})=4$, thereby yielding an instance of a GDD for which the monochromatic group chromatic number exceeds the chromatic number.

\begin{table}[tb]
\begin{center}
\caption{A 4-chromatic STS(21)}
\label{Table:STS21}
\vspace{2mm}
\footnotesize
\begin{tabular}{ccccccc}
\{0,3,9\} &
\{1,12,16\} &
\{2,8,19\} &
\{4,17,18\} &
\{5,6,14\} &
\{7,11,15\} &
\{10,13,20\} \\

\{0,1,2\} &
\{1,3,10\} &
\{2,3,11\} &
\{1,5,9\} &
\{1,4,11\} &
\{2,5,10\} &
\{0,5,11\} \\
\{2,4,9\} &
\{0,4,10\} &
\{3,4,5\} &
\{3,6,12\} &
\{4,6,13\} &
\{4,8,12\} &
\{4,7,14\} \\
\{5,8,13\} &
\{3,8,14\} &
\{5,7,12\} &
\{3,7,13\} &
\{6,7,8\} &
\{6,9,15\} &
\{7,9,16\} \\
\{8,9,17\} &
\{7,10,17\} &
\{8,11,16\} &
\{6,11,17\} &
\{8,10,15\} &
\{6,10,16\} &
\{9,10,11\} \\
\{9,12,18\} &
\{10,12,19\} &
\{11,12,20\} &
\{10,14,18\} &
\{11,14,19\} &
\{9,14,20\} &
\{11,13,18\} \\
\{9,13,19\} &
\{12,13,14\} &
\{0,12,15\} &
\{2,12,17\} &
\{1,14,15\} &
\{1,13,17\} &
\{2,14,16\} \\
\{0,14,17\} &
\{2,13,15\} &
\{0,13,16\} &
\{15,16,17\} &
\{3,15,18\} &
\{4,15,19\} &
\{5,15,20\} \\
\{4,16,20\} &
\{5,17,19\} &
\{3,17,20\} &
\{5,16,18\} &
\{3,16,19\} &
\{18,19,20\} &
\{0,6,18\} \\
\{1,6,19\} &
\{2,6,20\} &
\{1,8,18\} &
\{1,7,20\} &
\{0,8,20\} &
\{2,7,18\} &
\{0,7,19\} \\

\end{tabular}
\end{center}

\end{table}

However, it is not the case that $\chi_M({\cal H}_\pi) = 4$ for each parallel class $\pi$ of this STS(21).
Of the 130 parallel classes that it has, 70 of them have a monochromatic group chromatic number of 3 while the remaining 60 have
monochromatic group chromatic number 4.
We illustrate this distribution of chromatic numbers in Table~\ref{Table:Stats}, along with those
for the other STSs in the collection ${\collection}$ from Subsection~\ref{ssc:weakGDD} that contain parallel classes.
This table also presents statistics from an analysis of 16255 of the distinct parallel classes
of the 5-chromatic STS(63) of~\cite{FHW1994} (which has more than 4.2 trillion distinct parallel classes).

\begin{table}[htb]
\begin{center}
\caption{$\chi({\cal D}_\pi)$ and $\chi_M({\cal D}_\pi)$ for four 4-chromatic Steiner triple systems
and one 5-chromatic STS(63)}
\label{Table:Stats}
\begin{tabular}{|cccc|}
\hline
STS Order & No.\ of Parallel Classes & $\chi({\cal D}_\pi)$ & $\chi_M({\cal D}_\pi)$ \\ \hline
\multirow{2}{*}{21} & 70 & 3 & 3 \\
  & 60 & 3 & 4 \\ \hline
\multirow{2}{*}{27} & 284 & 3 & 3 \\
  & 341 & 3 & 4 \\ \hline
\multirow{3}{*}{33} & 85 & 3 & 3 \\
  & 15176 & 3 & 4 \\
  & 2885 & 4 & 4 \\ \hline
\multirow{4}{*}{39} & 547 & 3 & 3 \\
  & 248 & 3 & 4 \\
  & 1074213 & 4 & 4 \\
  & 144 & 4 & 5 \\ \hline
\multirow{2}{*}{63} & 16248 & 4 & 5 \\
  & 7 & 4 & 6 \\ \hline
\end{tabular}

\end{center}
\end{table}

Some questions naturally arise:

\begin{question}
	If ${\cal D}$ is a GDD,
	how large can the difference $\chi_M({\cal D}) - \chi({\cal D})$ be?
\end{question}

\begin{question}
	If ${\cal D}$ is a GDD,
	how large can the ratio $\frac {\chi_M({\cal D})}{\chi({\cal D})}$ be?
\end{question}

\subsection{Group-Equitable Colourings}
\label{Sec:GroupEquitable}

In contrast to colourings of GDDs with monochromatic groups, the 2-colourings of GDDs obtained by Horsley and Pike in \cite[Lemmas 2.3, 2.4 and 5.1]{horsley_pike} are such that each group of size $g$ has $\lfloor \frac{g}{2} \rfloor$ of its points coloured with one colour and the remaining $\lceil \frac{g}{2} \rceil$ points coloured with a second colour.
To generalise this idea, we call a weak $c$-colouring of a GDD in which each group of size $g$ has $\lfloor \frac{g}{c} \rfloor$ or $\lceil \frac{g}{c} \rceil$ points of each colour {\em group-equitable}.

As an initial result regarding group-equitable colourings, we consider transversal designs.
Recall that Theorem~\ref{thm:bounds_chrom} establishes that transversal designs are 2-chromatic (and moreover that they have monochromatic group chromatic number 2).

\begin{theorem}
	\label{thm:TDequiGroup}
	If $g \geq 4$ is an integer such that $k > \left\lceil\frac{g}{2}\right\rceil$, then every TD$(k+2,g)$ is group-equitably 2-colourable.
\end{theorem}

\begin{proof}
    Let the point set of the TD$(k+2,g)$, ${\cal D}$, be $V=\{1,2,\ldots,g\} \times \{1,2,\ldots,k+2\}$, and the groups be $G_j = \{ (1,j),(2,j),\ldots,(g,j) \}$ for $j \in \{1,2,\ldots,k+2\}$.
    Further we may relabel the points within the groups to ensure that the blocks through $(1,1)$ are of the form $B_s = \{ (1,1), (s,2), (s,3), (s,4), \ldots, (s,k+2) \}$ for each $s \in \{1,2,\ldots, g\}$.
	Let $C_1 = \{(i,j) : 1 \leq i \leq \left\lfloor\frac{g}{2}\right\rfloor, 1 \leq j \leq k+1\}
	\cup
	\{(i,k+2) : \left\lfloor\frac{g}{2}\right\rfloor < i \leq g \}$
	and let $C_2 = V \setminus C_1$.
	We claim that $C_1$ and $C_2$ comprise a suitable 2-colouring of ${\cal D}$.
    
    Firstly, note that for each $\left\lfloor \frac{g}{2} \right\rfloor < s \leq g$, $(s,k+1) \in B_s \cap C_2$ and $(s,k+2) \in B_{s} \cap C_1$. Similarly, for each $1\leq s \leq \left\lfloor\frac{g}{2}\right\rfloor $, $(s,k+1) \in B_{s} \cap C_1$ and $(s,k+2) \in B_{s} \cap C_2$.
	So no block containing $(1,1)$ is monochromatic.

    We now consider a block, $B= \{(t_j, j)\mid 1\leq j \leq k+2\}$, which does not contain $(1,1)$, and for a contradiction assume that it is monochromatic. 
    Considering the point of $B$ in the group $G_{k+2}$, suppose that 
    $(t_{k+2}, k+2) \in C_1$.
    Since each point of $B$ must be in $C_1$, we have $1 \leq t_{k+1} \leq \left\lfloor\frac{g}{2}\right\rfloor$.
    Considering $G_k$, we also have $1\leq t_{k} \leq \left\lfloor\frac{g}{2}\right\rfloor$, but since $(1,1)\not\in B$, $t_k\neq t_{k+1}$ (as otherwise the pair $\{(t_k,k),(t_{k+1},k+1)\}$ would appear in both $B$ and $B_{t_k}$), so there are $\left\lfloor\frac{g}{2}\right\rfloor-1$ choices for $t_k$. Proceeding in this manner, we see that at most $\left\lfloor\frac{g}{2}\right\rfloor$ of the remaining points of $B$ are in $C_1$.  Since $k > \left\lceil\frac{g}{2}\right\rceil$, this leads to a contradiction.

    We now consider the case $(t_{k+2}, k+2) \in C_2$.
    Since each point of $B$ must be in $C_2$, we have $\left\lfloor \frac{g}{2}\right\rfloor +1 \leq t_{k+1} \leq g$.
    Considering $G_k$, we also have $\left\lfloor \frac{g}{2}\right\rfloor 
    +1 \leq t_{k} \leq g$, but since $(1,1)\not\in B$, $t_k\neq t_{k+1}$ (as otherwise the pair $\{(t_k,k),(t_{k+1},k+1)\}$ would appear in both $B$ and $B_{t_k}$), so there are $\left\lceil \frac{g}{2}\right\rceil-1$ choices for $t_k$. Proceeding in this manner, we see that at most $\left\lceil \frac{g}{2}\right\rceil$ of the remaining points of $B$ are in $C_1$.  Since $k > \left\lceil\frac{g}{2}\right\rceil$, this leads to a contradiction.
\end{proof}

Observe that while the colouring described in Theorem~\ref{thm:TDequiGroup} results in a group-equitable colouring,
it is not block-equitable.  In particular, considering the blocks through $(1,1)$, we see that these blocks either have $k+1$ points of one colour and one of the other, or $k$ points of one colour and two of the other.
Because Theorem~\ref{thm:TDequiGroup} does not apply in the case when $g=4$ and $k=2$, we separately consider this scenario in the following lemma.

\begin{lemma}
	\label{cor:TD44}
	There exists a TD(4,4) which can be 2-coloured so that each group has two points of each colour.
\end{lemma}

\begin{proof}
We exhibit such a TD and a suitable colouring. Let the point set be $\mathbb{Z}_4 \times \mathbb{Z}_4$, with groups $G_i=\{i\} \times \mathbb{Z}_4$
for each $i \in \mathbb{Z}_4$.  The blocks are:
\[
\begin{array}{llll}
\{0_0,1_2,2_0,3_0\} & \{0_1, 1_2, 2_2, 3_2\} & \{0_2, 1_2, 2_1, 3_3\} & \{0_3, 1_2, 2_3, 3_1\} \\
\{0_0, 1_1, 2_2, 3_1\} & \{0_1, 1_1, 2_0, 3_3\} & \{0_2, 1_1, 2_3, 3_2\} & \{0_3, 1_1, 2_1, 3_0\} \\
\{0_0, 1_0, 2_1, 3_2\} & \{0_1, 1_0, 2_3, 3_0\} & \{0_2, 1_0, 2_0, 3_1\} & \{0_3, 1_0, 2_2, 3_3\} \\
\{0_0, 1_3, 2_3, 3_3\} & \{0_1, 1_3, 2_1, 3_1\} & \{0_2, 1_3, 2_2, 3_0\} & \{0_3, 1_3, 2_0, 3_2\}.
\end{array}
\]
Setting the colour classes to be 
\[
C_1=\{0_0,0_1,1_0,1_1,2_0,2_1,3_0,3_1\} \mbox{ and } C_2=\{0_2,0_3,1_2,1_3,2_1,2_2,3_2,3_3\}
\] 
gives the required colouring.   
\end{proof}

The following theorem shows that group-equitably $2$-colourable transversal designs can be used to construct other group-equitably $2$-colourable GDDs.

\begin{theorem}\label{thm:groupequitable}
	If there exists a $k$-GDD of type $h^u$, 
	and a TD$(k,g)$ such that the transversal design has a group-equitable 2-colouring in which exactly $\lfloor \frac{g}{2} \rfloor$ points in each group are assigned to one of the colour classes,
	then there exists a $k$-GDD of type $(gh)^u$ that can be group-equitably 2-coloured.
\end{theorem}

\begin{proof} 
Let ${\cal D} = (V,{\cal B},{\cal G})$ be a $k$-GDD of type $h^u$.  
	For each point $x \in V$, blow $x$ up into a set $\{x_0, x_2, \ldots, x_{g-1}\}$ of points, with $\lfloor \frac{g}{2} \rfloor$ points, say $x_0, x_1, \ldots, x_{\lfloor g/2 \rfloor-1}$, given colour 1 and the remaining $\lceil\frac{g}{2}\rceil$, namely $x_{\lfloor g/2 \rfloor}, \ldots, x_{g-1}$, colour 2.
	Let $B \in {\cal B}$ be a block. 
	Onto the set $B \times \{0,1,\ldots,g-1\}$
    place a TD$(k,g)$ whose points are aligned in agreement with the colouring of $B \times \{0,1,\ldots,g-1\}$. Note that, by permuting the points within each group if necessary, we may assume that the colour classes in the TD are consistent with this colouring.
\end{proof}

\begin{example}
    Taking $k=g=4$, we demonstrate the result of applying the process outlined in the proof of Theorem~\ref{thm:groupequitable} for the case $k=g=4$, using the $\TD(4,4)$ from Lemma~\ref{cor:TD44}.

We consider the {\em blowup} of a block $B=\{w,x,y,z\}$ to be the point set  $B \times \{0,1,2,3\}$.
Placing the $\TD(4,4)$ from Lemma~\ref{cor:TD44} on the blowup of $B$ yields the following blocks produced by $B$ in $\cal{D}'$.
\[
\begin{array}{llll}
\{{\color{blue}w_0},{\color{darkgreen}\boldsymbol{x_2}},{\color{blue}y_0},{\color{blue}z_0}\} & \{{\color{blue}w_1}, {\color{darkgreen}\boldsymbol{x_2}}, {\color{darkgreen}\boldsymbol{y_2}}, {\color{darkgreen}\boldsymbol{z_2}}\} & \{{\color{darkgreen}\boldsymbol{w_2}}, {\color{darkgreen}\boldsymbol{x_2}}, {\color{blue}y_1}, {\color{darkgreen}\boldsymbol{z_3}}\} & \{{\color{darkgreen}\boldsymbol{w_3}}, {\color{darkgreen}\boldsymbol{x_2}}, {\color{darkgreen}\boldsymbol{y_3}}, {\color{blue}z_1}\} \\
\{{\color{blue}w_0}, {\color{blue}x_1}, {\color{darkgreen}\boldsymbol{y_2}}, {\color{blue}z_1}\} & \{{\color{blue}w_1}, {\color{blue}x_1}, {\color{blue}y_0}, {\color{darkgreen}\boldsymbol{z_3}}\} & \{{\color{darkgreen}\boldsymbol{w_2}}, {\color{blue}x_1}, {\color{darkgreen}\boldsymbol{y_3}}, {\color{darkgreen}\boldsymbol{z_2}}\} & \{{\color{darkgreen}\boldsymbol{w_3}}, {\color{blue}x_1}, {\color{blue}y_1}, {\color{blue}z_0}\} \\
\{{\color{blue}w_0}, {\color{blue}x_0}, {\color{blue}y_1}, {\color{darkgreen}\boldsymbol{z_2}}\} & \{{\color{blue}w_1}, {\color{blue}x_0}, {\color{darkgreen}\boldsymbol{y_3}}, {\color{blue}z_0}\} & \{{\color{darkgreen}\boldsymbol{w_2}}, {\color{blue}x_0}, {\color{blue}y_0}, {\color{blue}z_1}\} & \{{\color{darkgreen}\boldsymbol{w_3}}, {\color{blue}x_0}, {\color{darkgreen}\boldsymbol{y_2}}, {\color{darkgreen}\boldsymbol{z_3}}\} \\
\{{\color{blue}w_0}, {\color{darkgreen}\boldsymbol{x_3}}, {\color{darkgreen}\boldsymbol{y_3}}, {\color{darkgreen}\boldsymbol{z_3}}\} & \{{\color{blue}w_1}, {\color{darkgreen}\boldsymbol{x_3}}, {\color{blue}y_1}, {\color{blue}z_1}\} & \{{\color{darkgreen}\boldsymbol{w_2}}, {\color{darkgreen}\boldsymbol{x_3}}, {\color{darkgreen}\boldsymbol{y_2}}, {\color{blue}z_0}\} & \{{\color{darkgreen}\boldsymbol{w_3}}, {\color{darkgreen}\boldsymbol{x_3}}, {\color{blue}y_0}, {\color{darkgreen}\boldsymbol{z_2}}\}.
\end{array}
\]

Applying the colouring given in Lemma~\ref{cor:TD44} yields the following colouring for the blowup.
\[
C_1=\{{\color{blue}w_0},{\color{blue}w_1},{\color{blue}x_0},{\color{blue}x_1},{\color{blue}y_0},{\color{blue}y_1},{\color{blue}z_0},{\color{blue}z_1}\} \mbox{ and } C_2=\{{\color{darkgreen}\boldsymbol{w_2}},{\color{darkgreen}\boldsymbol{w_3}},{\color{darkgreen}\boldsymbol{x_2}},{\color{darkgreen}\boldsymbol{x_3}},{\color{darkgreen}\boldsymbol{y_2}},{\color{darkgreen}\boldsymbol{y_3}},{\color{darkgreen}\boldsymbol{z_2}},{\color{darkgreen}\boldsymbol{z_3}}\}.
\]
We can repeat this process for each block in the GDD to get the 2-colouring $C_1=V\times\{0,1\}$ and $C_2=V\times\{2,3\}$.
\end{example}

Note that the placement of the TD$(k,g)$ in the proof of Theorem~\ref{thm:groupequitable} implies that the initial $k$-GDD, $\mathcal{D}$, of type $h^u$ is not embedded in the final $k$-GDD, $\mathcal{D'}$, of type $(gh)^u$.  Hence it is possible for $\mathcal{D}$ to have high chromatic number and yet $\chi(\mathcal{D}')=2$.

By Theorem~\ref{thm:TDequiGroup}, if there exists a TD$(k,g)$ with $g \geq 4$ and $k \geq \left\lceil\frac{g}{2}\right\rceil+3$, then it is group-equitably $2$-colourable.  We hence have the following corollary.

\begin{corollary}
    Suppose $g \geq 4$ is an integer 
    and $k \geq \left\lceil\frac{g}{2}\right\rceil+3$.
    If there exists a $k$-GDD of type $h^u$ and a TD$(k,g)$, then there exists a group-equitably $2$-colourable $k$-GDD of type $(gh)^u$.
\end{corollary}

In the particular case that $k=4$, necessary and sufficient conditions for the existence of a $4$-GDD$_{\lambda}$ of type $g^u$ have been characterised by Brouwer, Schrijver and Hanani~\cite{BSH1977}.  
\begin{theorem}[Theorem 6.3 of~\cite{BSH1977}]
	\label{lemma:BSH1977}
	The necessary and sufficient conditions for the existence of a 4-GDD$_\lambda$ of type $g^u$ are
	\begin{enumerate}
		\item $u \geq 4$,
		\item $\lambda (u-1) g \equiv 0$ $(${\rm mod} $3)$, and
		\item $\lambda u (u-1) g^2 \equiv 0$ $(${\rm mod} $12)$,
	\end{enumerate}
	except for the two definite exceptions of $(g,u,\lambda) \in \{ (2,4,1), (6,4,1) \}$.    
\end{theorem}
Using Theorem~\ref{thm:groupequitable}, we obtain the following existence result for group-equitably $2$-colourable group divisible designs with block size $4$.
\begin{corollary}\label{cor:groupequitable4gdd}
If $u \geq 4$, $(u-1)g \equiv 0 \pmod{3}$ and $u(u-1)g^2 \equiv 0 \pmod{12}$, then there exists a group-equitably $2$-colourable $4$-GDD of type $(4g)^u$ with a group-equitable $2$-colouring.
\end{corollary}

\begin{proof}
    Theorem~\ref{lemma:BSH1977} asserts the existence of a $4$-GDD of type $g^u$.  Using the TD$(4,4)$ from Lemma~\ref{cor:TD44} in Theorem~\ref{thm:groupequitable} gives the result. 
\end{proof}

Observe that several of the foregoing results establish the existence of group-equitably 2-colourable GDDs.  We note that if a GDD is group-equitably 2-colourable then it is necessarily group-equitably 2-chromatic.  The existence of group-equitable $c$-chromatic GDDs when $c \geq 3$ is left for future exploration.

\section{Acknowledgements}
The authors would like to gratefully acknowledge the following research support: Donovan, Kemp and Lefevre, 
Australian Research Council Centre of Excellence for Plant Success
in Nature and Agriculture (grant application CE200100015);
Burgess, NSERC (grant number RGPIN-2025-04633); Danziger, NSERC (grant number RGPIN-2022-03816);
Pike, NSERC (grant number RGPIN-2022-03829)
as well as computational support from the Centre for Analytics, Informatics and Research at Memorial University of Newfoundland.
Pike and  Yaz{\i}c{\i} also thankfully acknowledge Raybould Fellowship support from the University of Queensland.

We would especially like to thank Jeff Dinitz and the University of Vermont for hosting several of the authors during a research retreat.

\pagebreak
\section*{Appendix}
\renewcommand*{\thefootnote}{\fnsymbol{footnote}}

In Section~\ref{ssc:weakGDD} we discussed some properties of six specific 4-chromatic Steiner triple systems from the literature.
In Section~\ref{Sec:Monogroup} we mentioned some of these again, as well as a particular 5-chromatic STS(63).
Each of these seven systems is described in more detail below.

\begin{enumerate}[leftmargin=21mm]
\item[\textbf{STS(21)~}]
This STS was constructed by Haddad in~\cite{Haddad1999} as follows.
First, let $(\Integer_7,{\cal B})$ be the cyclic STS(7) developed from the base block $\{0,1,3\}$.
Then apply this tripling construction:  given a STS $(\Integer_v,{\cal B})$ of order $v$, 
let $(W,E({\cal B}))$ be the STS$(3v)$ where $W = \big\{x_s : x \in V, s \in \{1,2,3\} \big\}$
and for each $\{i,j,k\} \in {\cal B}$ the following triples are included in $E({\cal B})$:
\begin{center}
$\{ i_1, i_2, i_3 \}$,
$\{ j_1, j_2, j_3 \}$,
$\{ k_1, k_2, k_3 \}$,
$\{ i_1, j_1, k_1 \}$,
$\{ i_2, j_1, k_2 \}$,
$\{ i_3, j_1, k_3 \}$,\\
$\{ i_2, j_3, k_1 \}$,
$\{ i_2, j_2, k_3 \}$,
$\{ i_3, j_3, k_2 \}$,
$\{ i_1, j_3, k_3 \}$,
$\{ i_3, j_2, k_1 \}$,
$\{ i_1, j_2, k_2 \}$.    
\end{center}
An isomorphic copy of this STS(21) appears in our Table~\ref{Table:STS21}.

\item[\textbf{STS(25)~}]
This STS was given by de Brandes, Phelps and R\"{o}dl in~\cite{deBPR1982}.
It is a cyclic system that can be developed modulo 25 from the base blocks:  
$\{1, 2, 4\}$, $\{1, 5, 14\}$\footnote[1]{The starter block $\{1,5,14\}$ is mistyped as $\{1,5,24\}$ in~\cite{deBPR1982}.}, $\{1, 6, 12\}$, $\{1, 8, 18\}$.

\item[\textbf{STS(27)~}]
This STS was given by de Brandes, Phelps and R\"{o}dl in~\cite{deBPR1982}.
It is a cyclic system that can be developed modulo 27 from the base blocks:
$\{{1,2,4}\}$,
$\{{1,5,12}\}$,
$\{{1,6,18}\}$,
$\{{1,7,15}\}$,
$\{{1,10,19}\}$.

\item[\textbf{STS(33)~}]
This STS was given by de Brandes, Phelps and R\"{o}dl in~\cite{deBPR1982}.
It is a cyclic system that can be developed modulo 33 from the base blocks:
$\{{1,2,4}\}$,
$\{{1,5,15}\}$,
$\{{1,6,14}\}$,
$\{{1,7,19}\}$,
$\{{1,8,17}\}$,
$\{{1,12,23}\}$.

\item[\textbf{STS(37)~}]
This STS was given by de Brandes, Phelps and R\"{o}dl in~\cite{deBPR1982}.
It is a cyclic system that can be developed modulo 37 from the base blocks:
$\{{1,2,4}\}$,
$\{{1,5,15}\}$,
$\{{1,6,14}\}$,
$\{{1,7,22}\}$,
$\{{1,8,20}\}$,
$\{{1,10,21}\}$.

\item[\textbf{STS(39)~}]
This STS was given by Haddad in \cite{Haddad1999} using the same tripling construction described above for building the STS(21),
but now using the cyclic STS(13) having base blocks 
$\{0, 1, 4\}$ and  $\{0, 2, 7\}$
as the construction's ingredient.

\item[\textbf{STS(63)~}]
Consider the set of points of the projective 5-space PG(5,2), and the 1-dimensional subspaces of PG(5,2); these are
the point and block sets, respectively, of the projective STS(63).
That is, the non-zero vectors of $\Integer_2^6$ are the points of the system,
and any three non-zero vectors $x,y,z$ form a block $\{x,y,z\}$ if and only if their sum $x+y+z$ is the zero vector.
This STS was proved to be 5-chromatic by Fug\`ere, Haddad, and Wehlau~\cite{FHW1994}.

\end{enumerate}

In Section~\ref{ssc:weakGDD} we investigated $\chi({\cal D}_y)$ for each $\cal D$ that is one of the six above 4-chromatic Steiner triple systems.
The calculation of $\chi({\cal D}_y)$ for each point $y$ of $\cal D$ was performed by an exhaustive backtracking approach 
whereby all potential assignments of colours to the points of ${\cal D}_y$ using fewer than $\chi({\cal D}_y)$ colours were considered and found to not be valid colourings of the GDD.

In Sections~\ref{ssc:weakGDD} and~\ref{Sec:Monogroup} 
we investigated $\chi({\cal D}_\pi)$ for those cases where $\cal D$ is one of the above STS with order $v \in \{21,27,33,39,63\}$
(i.e., when the STS has the potential to admit a parallel class $\pi$).
Again, the calculation of the chromatic number was performed by an exhaustive backtracking approach that attempted to
find a $c$-colouring of ${\cal D}_\pi$ using increasingly many colours $c$, until a valid colouring was successfully found.
To identify the parallel classes of each STS,
we formulated an exact cover problem whereby the goal is to cover the point set of $\cal D$ (with each point being covered exactly once)
by a subset of the block set; any such subset of blocks is therefore a parallel class $\pi$ of $\cal D$.
This search for parallel classes was aided by the {\sf libexact} package made available by Kaski and Pottonen~\cite{KP08}.
As the order $v$ of the STS$(v)$ increases, the tasks of finding parallel classes and calculating chromatic numbers become increasingly time consuming;
for orders 39 and 63 we distributed the work over a computing cluster, but even then we found the case of $v=63$ to be infeasible to complete
(and so for $v=63$ we report partial results in Section~\ref{Sec:Monogroup}).


\begin{thebibliography}{99}

\bibitem{AdamsBryantLefevreWaterhouse2004} P.~Adams, D.~Bryant, J.~Lefevre and M.~Waterhouse.
Some equitably 3-colourable cycle decompositions,
{\em Disc. Math.} 284 (2004) 21--35.

\bibitem{AdamsBryantWaterhouse2007} P.~Adams, D.~Bryant and M.~Waterhouse.
Some equitably 2-colourable cycle decompositions, 
{\em Ars Combin.} 85 (2007), 49--64.

\bibitem{Beth} T.\ Beth.  Eine Bemerkung zur Absch\"{a}tzung der Anzahl orthogonaler lateinischer Quadrate mittels Siebverfahren, {\em Abh.\ Math.\ Sem.\ Hamburg} 53 (1983) 284--288.

\bibitem{deBPR1982}M.~de~Brandes, K.T.~Phelps and V.~R\"{o}dl.
Coloring Steiner triple systems,
{\em SIAM J.\ Alg.\ Disc.\ Meth.} 3 (1982) 241--249.

\bibitem{BSH1977}A.E.~Brouwer, A.~Schrijver and H.~Hanani.
Group divisible designs with block-size four,
{\em Discrete Math.} 20 (1977) 1--10.

\bibitem{BurgessMerola2021} A.~Burgess and F.~Merola. 
Equitably $2$-colourable even cycle systems,
{\em Australas.\ J.\ Combin.} 79 (2021), 437--453.

\bibitem{BurgessMerola2024} A.~Burgess and F.~Merola.
On equitably $2$-colourable odd cycle decompositions, 
{\em J.\ Combin.\ Des.} 32 (2024), 419--437.

\bibitem{Colbourn1982A} C.J.~Colbourn, M.J.~Colbourn, K.T.~Phelps and V.~R\"{o}dl. Coloring block designs is NP-complete, {\em SIAM J. Alg. Disc. Meth.} 3 (1982) 305--307.
\bibitem{Colbourn1982B} C.J.~Colbourn, M.J.~Colbourn, K.T.~Phelps and V.~R\"{o}dl.
Colouring Steiner quadruple systems, {\em Discrete Appl. Math.} 4 (1982) 103--111.


\bibitem{Handbook}C.J.~Colbourn, J.H.~Dinitz, editors.
{\em The CRC Handbook of Combinatorial Designs}, 2nd ed.
CRC Press Series on Discrete Mathematics and its Applications, Boca Raton, 2007.


\bibitem{STS19}C.J.~Colbourn, A.D.~Forbes, M.J.~Grannell, T.S.~Griggs, P.~Kaski, P.R.J.~\"{O}sterg{\aa}rd, D.A.~Pike and O.~Pottonen.
Properties of the Steiner triple systems of order 19,
{\em Electronic Journal of Combinatorics} 17(1) (2010) Research Paper \#R98.



\bibitem{CR1999} C.J.~Colbourn and A.~Rosa. Triple Systems. \emph{Oxford University Press} (1999).


\bibitem{FGLR2002}F.~Franek, T.S.~Griggs, C.C.~Lindner and A.~Rosa.
Completing the spectrum of 2-chromatic $S(2,4,v)$, {\em Discrete Math.}  247 (2002) 225--228.


\bibitem{FHW1994}J.~Fug\`{e}re, L.~Haddad and D.~Wehlau.
5-chromatic Steiner triple systems,
{\em J.\ Combin.\ Designs} 2 (1994) 287--299.

\bibitem{GDCM2020}R.~Gabrys, H.S.~Dau, C.J.~Colbourn, and O. Milenkovic. Set-Codes with Small Intersections and Small Discrepancies, {\em SIAM Journal on Discrete Mathematics.} 34 (2020) 1148--1171.

\bibitem{Haddad1999}L.~Haddad.  On the chromatic numbers of Steiner triple systems,
{\em J.\ Combin.\ Designs} 7 (1999) 1--10.



\bibitem{HLP1990}D.G.~Hoffman, C.C.~Lindner and K.T.~Phelps.
Blocking sets in designs with block size 4, {\em European J.\ Combin.} 11 (1990) 451--457.

\bibitem{HLP1991}D.G.~Hoffman, C.C.~Lindner and K.T.~Phelps.
Blocking sets in designs with block size four II, {\em Discrete Math.} 89 (1991) 221--229.




\bibitem{horsley_pike}D.~Horsley and D.A.~Pike.  On balanced incomplete block designs with specified weak chromatic number,
{\em J.\ Comb.\ Theory A} 123 (2014) 123--153.



\bibitem{KP08}
P.~Kaski and O.~Pottonen, {\sf libexact} User's Guide, Version 1.0,
Helsinki Institute for Information
Technology HIIT, HIIT Technical Reports 2008-1, 2008.


\bibitem{LefevreWaterhouse2005} J.~Lefevre and M.~Waterhouse.  
Some equitably $3$-colourable cycle decompositions of complete equipartite graphs, 
{\em Discrete Math.} 297 (2005), 60--77.

\bibitem{LiShen}S.~Li and H.~Shen.  Coloring of uniform group divisible designs with block size 3 or 4.  Unpublished manuscript.


\bibitem{Ling1999}A.C.H.~Ling.
On $2$-chromatic $(v,5,1)$-designs, {\em J.\ Geom.} 66 (1999) 144--148.

\bibitem{LutherPike2016}R.D.~Luther and D.A.~Pike.
Equitably Colored Balanced Incomplete Block Designs, {\em J.\ Combin.\ Designs} 24, (2016)  299--307.



\bibitem{MPR1983}R.A.~Mathon, K.T.~Phelps and A.~Rosa.
        Small Steiner Triple Systems and Their Properties,
        {\em Ars Combinatoria} 15 (1983) 3--110.
Errata, {\em Ars Combinatoria} 16 (1983) 286.

\bibitem{Mohacsy2011} H.~Moh\'{a}csy. The asymptotic existence of group divisible designs of large order with index one. {\em J. Comb. Theory A}  118, (2011) 1915--1924.



\bibitem{Pelikan1970}J.~Pelik\'{a}n.
Properties of balanced incomplete block designs,
{\em Combinatorial Theory and its Applications} III, Proc.\ Colloq.\ Balatonf\"{u}red 1969, North-Holland, Amsterdam, (1970) 869--889.

\bibitem{Phelps1984} K.T.~Phelps and V.~R\"{o}dl, On the algorithmic complexity of coloring simple
hypergraphs and Steiner triple systems. {\em Combinatorica} 4 (1984) 79--88. 

\bibitem{Rosa1970}A.~Rosa.
Steiner triple systems and their chromatic number,
{\em Acta Fac.\ Rerum Natur.\ Univ.\ Comenian.\ Math.} 24 (1970) 159--174.

\bibitem{RosaColbourn1992}A.~Rosa and C.J.~Colbourn.
Colorings of block designs, in {\em Contemporary Design Theory: A Collection of Surveys} (Eds. J.H.~ Dinitz and D.R.~Stinson), John Wiley \& Sons, New York (1992) 401--430.

\bibitem{Tabatabaei2020} S.K.~Tabatabaei, B.~Wang, N.B.M.~Athreya, B.~Enghiad, A.G.~Hernandez, C.J.~Fields, J.P.~Leburton, D.~Soloveichik, H.~Zhao, O.~Milenkovic. DNA punch cards for storing data on native DNA sequences via enzymatic nicking, {\em Nature Communications}  11(1) (2020) 1742.

\bibitem{Waterhouse2006} M.~Waterhouse.
Some equitably $2$-colourable cycle decompositions of complete multipartite graphs, 
{\em Util.\ Math.} 70 (2006), 201--220.

\bibitem{Woolbright1978}D.E.~Woolbright, An $n \times n$ Latin square has a transversal with at least $n - \sqrt{n}$ distinct symbols, {\em J. Comb. Theory A} 24 (1978) 235--237.

\bibitem{Yu2023}W.~Yu, Y.~Xi, X.~Wei, G.~Ge.
Balanced Set Codes With Small Intersections.
{\em IEEE Transactions on Information Theory} 69 (2023) 147--156.


\end{thebibliography}
\end{document}